%% file: MWspmain.tex
\documentclass[11pt,leqno]{amsart}

\usepackage[english]{babel}
\usepackage{url}
\usepackage[latin1]{inputenc}

\usepackage[all,2cell,emtex]{xy}
\usepackage{amssymb,amsmath,bbm,xr,a4wide,color,stmaryrd}
\usepackage[mathscr]{euscript}
\usepackage{mathrsfs}

\input{commandMWsp}

\title{The Milnor-Witt motivic ring spectrum and its associated theories}

\author{Fr\'ed\'eric D\'eglise}
\address{Institut Math\'ematique de Bourgogne - UMR 5584, Universit\'e de Bourgogne, 9 avenue Alain Savary, BP 47870, 21078 Dijon Cedex, France}
\email{frederic.deglise@ens-lyon.fr}
\urladdr{http://perso.ens-lyon.fr/frederic.deglise/}
 
\author{Jean Fasel}
\address{Institut Fourier - UMR 5582, Universit\'e Grenoble-Alpes, CS 40700, 38058 Grenoble Cedex 9, France}
\email{jean.fasel@gmail.com}
\urladdr{https://www.uni-due.de/~adc301m/staff.uni-duisburg-essen.de/Home.html}

\thanks{The first author was supported by the ANR (grant No. ANR-12-BS01-0002).}
\date{\today}

\newtheorem{thm}{Theorem}[subsection]
\newtheorem{prop}[thm]{Proposition}
\newtheorem*{propi}{Proposition}

\theoremstyle{remark} 
\newtheorem{rem}[thm]{Remark}

\newtheorem{ex}[thm]{Example}

\theoremstyle{definition} 
\newtheorem{df}[thm]{Definition}
\newtheorem{num}[thm]{}

\numberwithin{equation}{thm}

\newtheorem{thm*}{Theorem}

\begin{document}

\begin{abstract}
We build a ring spectrum representing Milnor-Witt motivic
 cohomology, as well as its \'etale local version and show how to
 deduce out of it three other theories: Borel-Moore homology,
 cohomology with compact support and homology.
 These theories, as well as the usual cohomology, are defined
 for signular schemes and satisfy the properties of their motivic analog
 (and more), up to considering more general twists.
 In fact, the whole formalism of these four theories can be functorially
 attached to any ring spectrum, giving finally maps between 
 the Milnor-Witt motivic ones to the classical motivic ones.
\end{abstract}

\maketitle

\setcounter{tocdepth}{3}
\tableofcontents

\section*{Introduction}

\input{introMWsp}

\section*{Special thanks}

\section*{Conventions} \label{conventions}

If $S$ is a base scheme, we will say \emph{$S$-varieties} for separated
 $S$-schemes of finite type.

We will simply call a symmetric monoidal category $\mathcal C$ monoidal. We generically denote by $\un$
 the unit object of a monoidal category. When this category
 depends on a scheme $S$, we also use the generic notation $\un_S$.

In the last section, 
 we will fix a perfect base field $k$
 and a coefficient ring $R$.
 
\section{Motivic homotopy theory and ring spectra}

\input{homotopy}

\section{The four theories associated with a ring spectrum}
\label{sec:theories}
\input{theories}

\section{The $\MW$-motivic ring spectrum}

\input{mwspectrum}

\bibliographystyle{amsalpha}
\bibliography{MWsp}

\end{document}

%% file: commandMWsp.tex
\newcommand{\HH}{\mathbf H}
\newcommand{\HM}{\mathbf H_{\mathrm{M}}}
\newcommand{\HMBM}{\mathbf H^{\mathrm{M},\mathrm{BM}}}
\newcommand{\HMx}[1]{\mathbf H_{\mathrm{M},#1}}
\newcommand{\Ht}{\mathbf H_{\mathrm{MW}}}
\newcommand{\HtBM}{\mathbf H^{\mathrm{MW},\mathrm{BM}}}
\newcommand{\Hlgt}{\mathbf H^{\mathrm{MW}}}
\newcommand{\HMet}{\mathbf H_{\mathrm{M},\et}}
\newcommand{\Htet}{\mathbf H_{\mathrm{MW},\et}}
\newcommand{\Htc}{\mathbf H_{\mathrm{MW},\mathrm{c}}}

\newcommand{\DMg}{\mathrm{DM}}
\newcommand{\DMcdh}{\mathrm{DM_{cdh}}}

\newcommand{\cupp}{{\scriptstyle \cup}}
\newcommand{\capp}{{\scriptstyle \cap}}

\newcommand{\T} {\mathscr T}
\newcommand{\E} {\mathbb E}
\newcommand{\F} {\mathbb F}

\DeclareMathOperator{\shg}{Sh}
\newcommand{\DAg}[1]{\Der_{\AA^1}(#1)}
\newcommand{\DAeg}[1]{\Der_{\AA^1}^{eff}(#1)}
\newcommand{\smg}{\mathrm{Sm}}

\newcommand{\un}{\mathbbm 1}

\newcommand{\piso}{\mathfrak p}

\newcommand{\MW}{\mathrm{MW}}

\newcommand{\DA}{\Der_{\AA^1}(k,R)}
\newcommand{\DAet}{\Der_{\AA^1,\et}(k,R)}
\newcommand{\DMt}{\widetilde{\mathrm{DM}}(k,R)}
\newcommand{\DM}{\mathrm{DM}(k,R)}
\newcommand{\SH}{\mathrm{SH}(k)}

\newcommand{\DMtx}[1]{\widetilde{\mathrm{DM}}_{#1}(k,R)}
\newcommand{\DMx}[1]{\mathrm{DM}_{#1}(k,R)}
\newcommand{\DAb}[1]{\Der_{\AA^1}(#1,R)}

\def\sKMW{\mathbf{K}^{\mathrm{M\hspace{-.2ex}W}}} 
\def\Wi{\mathbf{W}}

\def\H{\mathrm{H}}
\DeclareMathOperator{\Hom}{Hom}

\DeclareMathOperator{\uHom}{\underline{Hom}} 

\DeclareMathOperator{\wCH}{\widetilde{CH}}
\DeclareMathOperator{\CH}{CH}
\DeclareMathOperator{\spec}{Spec}

\DeclareMathOperator{\ilim}{\varinjlim}

\DeclareMathOperator{\Th}{\mathrm{Th}} 

\DeclareMathOperator{\Comp}{C}
\DeclareMathOperator{\K}{K}
\DeclareMathOperator{\Der}{D}

\newcommand{\derL}{\mathbf{L}}
\newcommand{\derR}{\mathbf{R}}

\DeclareMathOperator{\cK}{\mathcal K}

\newcommand{\ZZ} {\mathbb Z}

\renewcommand{\AA} {\mathbb A}
\newcommand{\PP} {\mathbb P}
\newcommand{\GG} {\mathbb{G}_m}

\newcommand{\zar}{\mathrm{Zar}}
\newcommand{\et}{\mathrm{\acute et}}

%% file: introMWsp.tex
One of the nice feature, and maybe part of the success
 of Voevodsky's theory of sheaves with transfers 
 and motivic complexes is to provide a rich cohomological theory.
 Indeed, apart from defining motivic cohomology, Voevodsky also
 obtains other theories: Borel-Moore motivic homology, 
 (plain) homology and motivic cohomology with compact support.
 These four theories are even defined for singular $k$-schemes
 and, when $k$ has characteristic $0$ (or with good resolution of
 singularities assumptions), they satisfy nice properties, such
 as duality. Finally, again in characteristic $0$,
 they can be identified with known theories:
 Borel-Moore motivic cohomology agrees with
 Bloch's higher Chow groups and homology (without twists) with
 Suslin homology. This is described
 in \cite[chap. 5, 2.2]{FSV}.\footnote{The results of Voevodsky
 have been generalized in positive characteristic $p$, after 
 taking $\ZZ[1/p]$-coefficients, using the results of Kelly \cite{Kel}.
 See \cite[Sec. 8]{CD5}.}
 
It is natural to expect that the Milnor-Witt version of motivic complexes
 developed in this book allows one to get a similar formalism. That is what
 we prove in this chapter. However, we have chosen a different path
from Voevodsky's, based on the tools at our disposal nowadays.
 In particular, we have not developed a theory of relative 
 $\mathrm{MW}$-cycles, nevertheless an interesting problem which we leave for
 future work. Instead, we rely on two classical theories, that of
 ring spectra from algebraic topology and that of Grothendieck' six
 functors formalism established by Ayoub in $\AA^1$-homotopy along the
 lines of Voevodsky.
 
In fact, it is well known since the axiomatization of Bloch
 and Ogus (see \cite{BO}) that a corollary of the six functors formalism
 is the existence of twisted homology and cohomology related by duality.
 In the present work, we go further that Bloch and Ogus in several directions. 
 First, we establish
 a richer structure, made of Gysin morphisms (wrong-way functoriality),
 and establish more properties, such as cohomological descent (see below for
 a precise description). Second, following Voevodsky's motivic picture,
 we show that the pair of theories, cohomology and Borel-Moore homology,
 also comes with a compactly supported pair of theories, cohomology
 with compact support and (plain) homology. And finally we consider
 more general twists for our theories, corresponding to the tensor product
 with the Thom space of a vector bundle.

Indeed, our main case of interest, $\mathrm{MW}$-motivic cohomology,
 is an example of non orientable cohomology theory. In practice,
 this means there are no Chern classes, and therefore no Thom classes
 so that one cannot identify the twists by a Thom space with a twist
 by an integer.\footnote{Recall that given an oriented cohomology
 theory $\E^{**}$ and a vector bundle $V/X$ of rank $n$,
 the Thom class of $V/X$ provides an isomorphism
 $\E^{**}(\Th(E)) \simeq \E^{*-2n,*-n}(X)$. See \cite{Panin03} for
 the cohomological point of view or \cite{Deg12} for the ring spectrum
 point of view.} In this context, it is especially relevant to take care
 of a twist of the cohomology or homology of a scheme $X$
 by an arbitrary vector bundle over $X$, and in fact by a virtual
 vector bundle over $X$.\footnote{See Paragraph \ref{num:basic_Thom}
 for this notion. 
 If one is interested in cohomology/homology up to isomorphism, then
 one can take instead of a virtual vector bundle a class in the K-theory
 ring $K_0(X)$ of $X$ (see Remark \ref{rem:twists&K0}).}
 Our first result is to build a ring spectrum $\Ht R$
 (Definition \ref{df:M&MW_ring_sp})
 over a perfect field $k$, with coefficients in an arbitrary ring $R$, which represents
 $\mathrm{MW}$-motivic cohomology (see Paragraph \ref{num:MW-sp&concrete_coh}
 for precise statements).
 In this text, the word \emph{ring spectrum} (for instance the object $\Ht R$)
 means a commutative monoid object of the $\AA^1$-derived category
 $\DAg k$ over a field $k$ (see Section \ref{sec:6functors} for
 reminders). One can easily get from this one a commutative monoid
 of the stable homotopy category $\SH$ (see Remark \ref{rem:on_spectra}).
 Out of the ring spectrum $\Ht R$, we get as expected the following
 theories, associated with an arbitrary $k$-variety $X$ (i.e. a separated $k$-scheme of finite type) 
 and a pair $(n,v)$ where $n \in \ZZ$, $v$ is a virtual vector
 bundle over $X$:
\begin{itemize}
\item $\mathrm{MW}$-motivic cohomology, $\Ht^n(X,v,R)$,
\item $\mathrm{MW}$-motivic Borel-Moore homology, $\HtBM_n(X,v,R)$,
\item $\mathrm{MW}$-motivic cohomology with compact support, $\Htc^n(X,v,R)$,
\item $\mathrm{MW}$-motivic homology, $\Hlgt_n(X,v,R)$.
\end{itemize}
As expected, one gets the following computations of
 $\mathrm{MW}$-motivic cohomology.
\begin{propi}
Assume $k$ is a perfect field, $X$ a smooth $k$-scheme and
 $(n,m)$ a pair of integers.

Then there exists a canonical identification:
$$
\Ht^n(X,m,R)
 \simeq
 \Hom_{\DMt}\big(\tilde M(X),\un(m)[n+2m]\big)=\Ht^{n+2m,m}(X,R)
$$
using the notations of \cite{DF1}.

If in addition, we assume $k$ is infinite, then one has the following
 computations:
$$
\Ht^n(X,m,\ZZ)=\begin{cases}
\wCH^m(X) & \text{if } n=0, \\
\H^n_\zar(X,\sKMW_0) & \text{if } m=0, \\
\H^{n+2m}_\zar(X,\tilde\ZZ(m)) & \text{if } m>0, \\
\H^{n+m}_\zar(X,\Wi) & \text{if } m<0.
\end{cases}
$$
\end{propi}
In fact, the first identification follows from the definition
 and basic adjunctions (see Paragraph \ref{num:trivial_comput_coh})
 while the second one was proved in \cite{DF1} 
 (as explained in \ref{num:MW-sp&concrete_coh}).
 
In fact, though the ring spectrum $\Ht R$ is our main case of interest,
 because of the previous computation,
 the four theories defined above,
 as well as their properties that we will give below,
 are defined for an arbitrary ring spectrum $\E$ --- and indeed, this
 generality is useful as will be explained afterwards.
 The four theories associated with $\E$ enjoy
 the following functoriality properties (Section \ref{sec:functoriality})
 under the following assumptions:
\begin{itemize}
\item $f:Y \rightarrow X$ is an arbitrary morphism of $k$-varieties
 (in fact $k$-schemes for cohomology);
\item $p:Y \rightarrow X$ is a morphism of $k$-varieties which is either
 smooth or such that
 $X$ and $Y$ are smooth over $k$.
 In any case, $p$ is a local complete intersection morphism
 and one can define its virtual tangent bundle denoted by $\tau_p$.
\end{itemize}

\hspace{-2.2cm}
{\renewcommand{\arraystretch}{1.5}
\begin{tabular}{|c|p{2cm}|c|p{2cm}|c|}
\hline
&\multicolumn{2}{|c|}{natural variance}
 & \multicolumn{2}{c|}{Gysin morphisms} \\
\hline
Theory & additional assumption on $f$ &  & additional assumption on $p$ &  \\
\hline
cohomology & none & $\E^n(X,v) \xrightarrow{f^*} \E^n(Y,f^{-1}v)$
 & proper & $\E^n(Y,p^{-1}v+\tau_p) \xrightarrow{p_*} \E^n(X,v)$ \\
\hline
BM-homology & proper & $\E_n^{BM}(Y,f^{-1}v) \xrightarrow{f_*} \E_n^{BM}(X,v)$
 & none & $\E_n^{BM}(X,v) \xrightarrow{p^*} \E_n^{BM}(Y,p^{-1}v-\tau_p)$ \\
\hline
c-cohomology & proper & $\E^n_c(X,v) \xrightarrow{f^*} \E^n_c(Y,f^{-1}v)$
 & none & $\E^n(Y,p^{-1}v+\tau_p) \xrightarrow{p_*} \E^n(X,v)$ \\
\hline
homology & none & $\E_n(Y,f^{-1}v) \xrightarrow{f_*} \E_n(X,v)$
 & proper & $\E_n(X,v) \xrightarrow{p^*} \E_n(Y,p^{-1}v-\tau_p)$ \\
\hline
\end{tabular}}
\medskip

Given a closed immersion $i:Z \rightarrow X$ between arbitrary $k$-varieties
 with complement open immersion $j:U \rightarrow X$,
 and a virtual vector bundle $v$ over $X$,
 there exists the so-called \emph{localization long exact sequences}
 (Paragraph \ref{num:localization_BM&c}):
\begin{align*}
\E^{BM}_n(Z,i^{-1}v)
& \xrightarrow{i_*} \E^{BM}_n(X,v)
 \xrightarrow{j^*} \E^{BM}_n(U,j^{-1}v)
 \rightarrow \E^{BM}_{n-1}(Z,i^{-1}v), \\
\E_c^n(U,j^{-1}v)
& \xrightarrow{j_*} \E_c^n(X,v)
 \xrightarrow{i^*} \E_c^n(Z,i^{-1}v)
 \rightarrow \E_c^{n+1}(U,j^{-1}v).
\end{align*}

There exists the following products for a $k$-variety $X$ (or simply a 
 $k$-scheme in the case of cup-products), and couples $(n,v)$, $(m,w)$
 of integers and virtual vector bundles on $X$:
 
\medskip
{\renewcommand{\arraystretch}{2}
\begin{tabular}{|c|c|c|}
\hline
Name & pairing & symbol \\
\hline
cup-product
 & $\E^n(X,v) \otimes \E^m(X,w) \rightarrow \E^{n+m}(X,v+w)$
 & $x \cupp y$ \\
\hline
cap-product 
 & $\E_n^{BM}(X,v) \otimes \E^m(X,w) \rightarrow \E^{BM}_{n-m}(X,v-w)$
 & $x \capp y$ \\
\hline
cap-product with support
 & $\E_n^{BM}(X,v) \otimes \E^m_c(X,w) \rightarrow \E_{n-m}(X,v-w)$
 & $x \capp y$ \\
\hline
\end{tabular}}
\medskip

Given a smooth $k$-scheme $X$ with tangent bundle $T_X$ there exists
 (Definition \ref{df:fdl_class})
 a \emph{fundamental class} $\eta_X \in \E_0^{BM}(X,T_X)$
 such that the following morphisms,  called the \emph{duality isomorphisms},
\begin{align*}
\E^n(X,v) &\rightarrow \E^{BM}_{-n}(X,T_X-v),
 x \mapsto \eta_X \capp y \\
\E^n_c(X,v) &\rightarrow \E_{-n}(X,T_X-v),
 x \mapsto \eta_X \capp y
\end{align*}
are isomorphisms (Theorem \ref{thm:duality}).

Finally, the four theories satisfy the following descent properties.
 Consider a cartesian square 
$$
\xymatrix@=10pt{
W\ar^k[r]\ar_g[d]\ar@{}|\Delta[rd] & V\ar^f[d] \\
Y\ar_i[r] & X
}
$$
of $k$-varieties (or simply $k$-schemes in the case of cohomology
 (we refer the reader to Paragraph \ref{num:basic_6functors}
 for the definition of Nisnevich and $\mathrm{cdh}$ distinguished
 applied to $\Delta$).
 Put $h=i \circ g=f \circ k$.
 We let $v$ be a virtual vector bundle over $X$.
 Then:

\medskip
\hspace{-1.8cm}
{\renewcommand{\arraystretch}{2}
\begin{tabular}{|c|c|}
\hline
Assumption on $\Delta$ & descent long exact sequence \\
\hline
Nisnevich or cdh
 & $\E^n(X,v)
 \xrightarrow{i^*+f^*} \E^n(Y,i^{-1}v) \oplus \E^n(V,f^{-1}v)
 \xrightarrow{k^*-g^*} \E^n(W,h^{-1}v)
 \rightarrow \E^{n+1}(X,v)$ \\
 & $\E_n(W,h^{-1}v)
 \xrightarrow{k_*-g_*} \E_n(Y,i^{-1}v) \oplus \E_n(V,f^{-1}v)
 \xrightarrow{i_*+f_*} \E_n(X,v)
 \rightarrow \E_{n-1}(X,v)$ \\
\hline
Nisnevich
 & $\E^{BM}_n(X,v)
 \xrightarrow{i^*+f^*} \E^{BM}_n(Y,i^{-1}v) \oplus \E^{BM}_n(V,f^{-1}v)
 \xrightarrow{k^*-g^*} \E^{BM}_n(W,h^{-1}v)
 \rightarrow \E^{BM}_{n-1}(X,v)$ \\
 & $\E_c^n(W,h^{-1}v)
 \xrightarrow{k_*-g_*} \E_c^n(Y,i^{-1}v) \oplus \E_c^n(V,f^{-1}v)
 \xrightarrow{i_*+f_*} \E_c^n(X,v)
 \rightarrow \E_c^{n+1}(X,v)$ \\
\hline
$\mathrm{cdh}$
 & $\E^{BM}_n(W,h^{-1}v)
 \xrightarrow{k_*-g_*} \E^{BM}_n(Y,i^{-1}v) \oplus \E^{BM}_n(V,f^{-1}v)
 \xrightarrow{i_*+f_*} \E^{BM}_n(X,v)
 \rightarrow \E^{BM}_{n-1}(X,v)$ \\
 & $\E^n_c(X,v)
 \xrightarrow{i^*+f^*} \E^n_c(Y,i^{-1}v) \oplus \E^n_c(V,f^{-1}v)
 \xrightarrow{k^*-g^*} \E^n_c(W,h^{-1}v)
 \rightarrow \E^{n+1}_c(X,v)$ \\
\hline
\end{tabular}}
\medskip

As the reader can guess, the four theories are functorial in the ring
 spectrum $\E$. And in fact, out of the construction of \cite{DF1},
 one gets a canonical morphism of ring spectra
 (Paragraph \ref{num:MW-regulators}):
$$
\varphi:\Ht R \rightarrow \HM R
$$
from the $\mathrm{MW}$-motivic ring spectrum to Voevodsky's motivic
 ring spectrum (also simply called the motivic Eilenberg-MacLane
 spectrum), the two ring spectra considered over $k$ and with coefficients
 in an arbitrary ring of coefficients $R$.\footnote{Interestingly, this map 
 can be compared with Beilinson's classical regulator.}
 
 Therefore, one deduces natural maps, compatible with the functorialities
 described above, from the four $\mathrm{MW}$-motivic theories
 to their motivic analog, which can be identified with the versions
 defined by Voevodsky in \cite[chap. 5]{FSV} when the characteristic
 exponent of $k$ is invertible in $R$ (as we recall in paragraphs
 \ref{num:mot_coh} and \ref{num:BM_motivic}).

Note finally that we also construct the \'etale analog of 
 the $\mathrm{MW}$-motivic and motivic ring spectra, 
 which are linked with their classical (Nisnevich) counterparts by
 canonical morphisms (see again \ref{num:MW-regulators}).

\section*{Plan of the paper}

As said previously,
 this paper is an application of our previous work and of general
 motivic $\AA^1$-homotopy. So we have tried to give complete reminders
 for a non specialist reader.
 
In Section 1, we first recall the formalism of the $\AA^1$-derived
 category, as introduced by Morel, and the associated six functors
 formalism as constructed by Ayoub following Voevodsky.
 Then we give a brief account of the theory of ring spectra,
 specialized in the framework of the $\AA^1$-derived category.

In Section 2, we construct the four theories associated with
 an arbitrary ring spectrum and establish the properties listed 
 above.
 
Finally in Section 3, we apply these results to the particular case
 of $\mathrm{MW}$-motivic cohomology, as well as its \'etale version,
 and seemingly the classical case of motivic cohomology.
 We consider in more details the case of cohomology and Borel-Moore
 homology, and conclude this paper with the canonical maps
 relating these ring spectra.

%% file: homotopy.tex
\subsection{Reminder on Grothendieck's six functors formalism}

\label{sec:6functors}

\begin{num}
Let us fix a base scheme $S$.
 We briefly recall the construction of Morel's $\PP^1$-stable
 and $\AA^1$-derived category over $S$ using \cite{CD3} as a reference text.
 The construction has also been recalled in \cite{DF1} in the particular case
 where $S$ is the spectrum of a (perfect) field.

Let $\shg(S)$ be the category of Nisnevich sheaves of abelian groups over
 the category of smooth $S$-schemes $\smg_S$.
 This is a Grothendieck abelian category\footnote{Recall that an abelian category is Grothendieck abelian
  if it admits small coproducts, a family of generators and filtered colimits are exact.
  As usual, one gets that the category $\shg(S)$ is Grothendieck abelian
  from the case of presheaves and the adjunction $(a,\mathcal O)$ where $a$ is
  the associated sheaf functor, $\mathcal O$ the obvious forgetful functor.}.
 Given a smooth $S$-scheme $X$, one denotes by $\ZZ_S(X)$ the Nisnevich sheaf associated
 with the presheaf $Y \mapsto \ZZ[\Hom_S(Y,X)]$. The essentially small family $\ZZ_S(X)$
 generates the abelian category $\shg(S)$. The category $\shg(X)$ admits a closed 
 monoidal structure, whose tensor product is defined by the formula:
\begin{equation}\label{eq:tensor}
F \otimes G=\ilim_{X/F,Y/G} \ZZ_S(X \times_S Y)
\end{equation}
where the limit runs over the category whose objects are couples of morphisms
 $\ZZ_S(X) \rightarrow F$ and $\ZZ_S(Y) \rightarrow G$ and morphisms are given
 by couples $(x,y)$ fitting into commutative diagrams of the form:
$$
\xymatrix@=10pt{
\ZZ_S(X)\ar[rd]\ar^x[rr] && \ZZ_S(X'),\ar[ld]
 & \ZZ_S(Y)\ar[rd]\ar^y[rr] && \ZZ_S(Y').\ar[ld] \\
& F & & & G &
}
$$
 Note in particular
 that $\ZZ_S(X) \otimes \ZZ_S(Y)=\ZZ_S(X \times_S Y)$. We let the reader check that this definition coincides with the one given in \cite[1.2.14]{DF1} when $S=k$.

According to \cite{CD3}, 
 the category $\Comp(\shg(S))$ of complexes with coefficients in $\shg(S)$ admits 
 a monoidal model structure whose weak equivalences are quasi-isomorphisms and:
\begin{itemize}
\item the class of \emph{cofibrations} is given by the smallest class
 of morphisms of complexes closed under suspensions, pushouts,
 transfinite compositions and retracts generated by the inclusions
\begin{equation}
\ZZ_S(X) \rightarrow C\big(\ZZ_S(X) \xrightarrow{Id} \ZZ_S(X)\big)[-1]
\end{equation}
for a smooth $S$-scheme $X$.
\item \emph{fibrations} are given by the epimorphisms of complexes whose
 kernel $C$ satisfies the classical \emph{Brown-Gersten property}:
 for any cartesian square of smooth $S$-schemes
$$
\xymatrix@=10pt{
W\ar^k[r]\ar_q[d] & V\ar^p[d] \\
U\ar_j[r] & X
}
$$
such that $j$ is an open immersion, $p$ is \'etale and induces an isomorphism of
 schemes $p^{-1}(Z) \rightarrow Z$ where $Z$ is the complement of $U$ in $X$ endowed with its reduced subscheme structure, the resulting square of complexes of abelian
 groups
$$
\xymatrix@=10pt{
C(X)\ar[r]\ar[d] & C(V)\ar[d] \\
C(U)\ar[r] & C(W)
}
$$
is homotopy cartesian.
\end{itemize}
Indeed, from Examples 2.3 and 6.3 of \emph{op. cit.}, one gets a descent structure
 $(\mathcal G,\mathcal H)$
 (Def. 2.2 of \emph{op. cit.}) on $\shg(S)$
 where $\mathcal G$ is the essential family of generators constituted by the sheaves $\ZZ_S(X)$
 for $X/S$ smooth, and $\mathcal H$ is constituted by the complexes of the form
$$
0 \rightarrow \ZZ_S(W) \xrightarrow{q_*-k_*} \ZZ_S(U) \oplus \ZZ_S(V)
 \xrightarrow{j_*+p_*} \ZZ_S(X) \rightarrow 0.
$$
This descent structure is flat (\textsection 3.1 of \emph{loc. cit.})
 so that 2.5, 3.5, 5.5 gives
 the assertion about the monoidal model structure described above.
 Note moreover that this model structure is proper, combinatorial and satisfies
 the monoid axiom.
\end{num}

\begin{rem} \label{rem:compact_gen}
The descent structure defined in the preceding paragraph is also bounded
 (\textsection 6.1 of \emph{loc. cit.}). This implies in particular that
 the objects $\ZZ_S(X)$, as complexes concentrated in degree $0$,
 are compact in
 the derived category $\Der(\shg(S))$ (see Th. 6.2 of \emph{op. cit.}).
 Moreover, one can describe explicitly the subcategory of $\Der(\shg(S))$ 
 generated by these objects
 (see \emph{loc. cit.}).
\end{rem}

\begin{num}
Recall now that we get the $\AA^1$-derived category
 by first $\AA^1$-localizing the model category $\Comp(\shg(S))$,
 which amounts to invert in its homotopy category $\Der(\shg(S))$
 morphisms of the form
$$
\ZZ_S(\AA^1_X) \rightarrow \ZZ_S(X)
$$
for any smooth $S$-scheme $X$. One gets the so called $\AA^1$-local
 Nisnevich descent model structure (cf. \cite[5.2.17]{CD3}),
 which is again proper monoidal. One denotes by $\DAeg S$
 its homotopy category.
 Then one stabilizes the latter model category with respect to
 Tate twists, or equivalently with respect to the object:
$$
\un_S\{1\}=\mathrm{coker}\big(\ZZ_S(\{1\}) \rightarrow \ZZ_S(\GG)\big).
$$
This is based on the use of symmetric spectra (cf. \cite[5.3.31]{CD3}),
 called in our particular case \emph{Tate spectra}.
 The resulting homotopy category, denoted by $\DAg S$ is triangulated monoidal
 and is characterized by the existence of an adjunction of triangulated
 categories
$$
\Sigma^\infty:\DAeg S \leftrightarrows \DAg S:\Omega^\infty
$$
such that $\Sigma^\infty$ is monoidal and the object
 $\Sigma^\infty(\ZZ_S\{1\})$ is $\otimes$-invertible 
 in $\DAg S$. As usual, one denotes by $K\{i\}$ the tensor product
 of any Tate spectrum $K$ with the $i$-th tensor power of
 $\Sigma^\infty(\ZZ_S\{1\})$.
 Besides, we also use the more traditional Tate twist:
$$
\un_S(1)=\un_S\{1\}[-1].
$$
\end{num}

\begin{rem}\label{rem:generators}
 Extending Remark \ref{rem:compact_gen},
 let us recall that the Tate spectra of the form
 $\Sigma^\infty \ZZ_S(X)\{i\}$, $X/S$ is smooth and $i \in \ZZ$,
 are compact and form a family of generators
 of the triangulated category $\DAg S$ in the sense that
 every object of $\DAg S$ is a homotopy colimit of spectra
 of the preceding form (see \cite[5.3.40]{CD3}).
\end{rem}

\begin{num}\label{num:basic_Thom}
\textit{Thom spaces of virtual bundles}.--
It is important in our context to introduce more general twists
 (in the sense of \cite[D\'efinition 1.1.39]{CD3}).
 Given a base scheme $X$, and a vector bundle $V/X$ one defines the Thom
 space associated with $V$, as a Nisnevich sheaf over $\smg_X$,
 by the following formula:
$$
\Th(V)=\mathrm{coker}\big(\ZZ_X(V^\times) \rightarrow \ZZ_X(V)\big),
$$
where $V^\times$ denotes the complement in $V$ of the zero section.
 Seen as an object of $\DAg X$, which we still denote by $\Th(V)$,
 it becomes $\otimes$-invertible
 --- as it is locally of the form $\Th(\AA^n_X) \simeq \un_X(n)[2n]$.
 Therefore, we get a functor
$$
\Th:\mathrm{Vec}(X) \rightarrow \mathrm{Pic}(\DAg X)
$$
from the category of vector bundles over $X$ to that of
 $\otimes$-invertible Tate spectra. 
According to \cite[4.1.1]{Riou},
 given any exact sequence of vector bundles:
\begin{equation*}\tag{$\sigma$}
0 \rightarrow V' \rightarrow V \rightarrow V'' \rightarrow 0,
\end{equation*}
one gets a canonical isomorphism
$$
\Th(V) \xrightarrow{\epsilon_\sigma} \Th(V') \otimes \Th(V'')
$$
allowing to canonically extend the preceding functor
 to a functor from the category $\mathcal K(X)$ 
 of virtual vector bundles over $X$ introduced in \cite[\S 4]{Deligne87}
 to $\otimes$-invertible objects of $\DAg X$
$$
\Th:\mathcal K(X) \rightarrow \mathrm{Pic}(\DAg X)
$$
sending direct sums to tensors products.
\end{num}

\begin{rem}
The isomorphism classes of objects of $\mathcal K(X)$
 gives the K-theory ring $K_0(X)$ of $X$. In other words,
 neglecting morphisms, the construction recalled above associates
 to any element $v$ of $K_0(X)$ a canonical isomorphism class of 
 Tate spectra $\Th(v)$ which satisfies the relation:
 $\Th(v+w)=\Th(v) \otimes \Th(w)$.
\end{rem}

\begin{num}\label{num:basic_funct}
Let us finally recall the basic functoriality satisfied by sheaves
 and derived categories introduced previously.

Let $f:T \rightarrow S$ be a morphism of schemes.
 We get a morphism of sites
 $f^{-1}:\smg_S \rightarrow \smg_T$ defined by  $X/S \mapsto (X \times_S T/T)$
 and therefore an adjunction of categories of abelian Nisnevich sheaves:
\begin{equation}\label{eq:pullback}
f^*:\shg(S) \rightarrow \shg(T):f_*
\end{equation}
such that $f_*(G)=G \circ f^{-1}$ and
$$
f^*(F)=\ilim_{X/F} \ZZ_T(X \times_S T)
$$
where the colimit runs over the category of morphisms $\ZZ_S(X) \rightarrow F$.
 Recall that $f^*$ is not exact in general.

If in addition $p=f$ is smooth, one gets another morphism of sites:
$$
p_\sharp:\smg_T \rightarrow \smg_S,
 (Y \rightarrow T) \mapsto (Y \rightarrow T  \xrightarrow p S).
$$
One can check that $p^*(F)=F \circ p_\sharp$ and we get an adjunction
 of additive categories:
$$
p_\sharp:\shg(T) \rightarrow \shg(S):p^*
$$
such that:
\begin{equation}\label{eq:pushforward}
p_\sharp(G)=\ilim_{Y/G} \ZZ_S(Y \rightarrow T \xrightarrow f S),
\end{equation}
this time the colimit runs over the category of morphisms
 $\ZZ_T(Y) \rightarrow G$.

Using formulas \eqref{eq:tensor}, \eqref{eq:pullback}
 and \eqref{eq:pushforward}, 
 one can check the following basic properties:
\begin{enumerate}
\item \textit{Smooth base change formula}.-- For any cartesian square
 of schemes
$$
\xymatrix@=10pt{
Y\ar^q[r]\ar_g[d] & X\ar^f[d] \\
T\ar^p[r] & S
}
$$
such that $p$ is smooth, the canonical map
$$
q_\sharp g^* \rightarrow f^*p_\sharp
$$
is an isomorphism.
\item \textit{Smooth projection formula}.-- 
 For any smooth morphism $p:T \rightarrow S$
 and any Nisnevich sheaves $G$ over $T$ and $F$ over $S$,
 the canonical morphism:
$$
p_\sharp(G \otimes p^*(F)) \rightarrow p_\sharp(G) \otimes F
$$
is an isomorphism.
\end{enumerate}
We refer the reader to \cite[1.1.6, 1.1.24]{CD3} for the definition
 of the above canonical maps. Note that the properties stated above
 shows that $\shg$ is a $\smg$-premotivic abelian category
 in the sense of \cite[1.4.2]{CD3} (see also \cite[Ex. 5.1.4]{CD3}).

According to the theory developed in \cite[\textsection 5]{CD3},
 the adjunctions $(f^*,f_*)$ and $(p_\sharp,p^*)$ for $p$ smooth
 can be derived and induce triangulated functors
\begin{align*}
\derL f^*:\DAg S & \leftrightarrows \DAg T:\derR f_*, \\
\derL p_\sharp:\DAg S & \leftrightarrows \DAg T:p^*.
\end{align*}
 By abuse of notations,
 we will simply denote the derived functors by $f^*, f_*, p_\sharp$.
 Then the analogues of smooth base change
 and smooth projection formulas stated above hold
 (see \cite[Ex. 5.3.31]{CD3}). In other words, we get a premotivic
 triangulated category (cf. \cite[1.4.2]{CD3}) which by construction
 satisfies the homotopy and stability relation (\cite[2.1.3, 2.4.4]{CD3}).
\end{num}
\begin{df}
Consider the notations of \ref{num:basic_Thom} and \ref{num:basic_funct}.

Let $S$ be a base scheme, $X/S$ a smooth scheme,
 and $v$ be a virtual vector bundle over $X$. One defines the Thom
 space of $v$ above $S$ as the object:
$$
\Th_S(v)=p_\sharp(\Th(v)).
$$
\end{df}
Of course, unless $X=S$, $\Th_S(v)$ is in general not $\otimes$-invertible
 and we do not have the relation:
 $\Th_S(v\oplus w)=\Th_S(v) \otimes \Th_S(w)$.

\begin{num}\label{num:localization}
Consider again the notations of paragraph \ref{num:basic_funct}.
One can check the so-called localization property
 for the fibered category $\DAg -$ (cf. \cite[2.4.26]{CD3}):
 for any closed immersion $i:Z \rightarrow S$ with complement
 open immersion $j:U \rightarrow S$, and any Tate spectrum $K$ over $S$,
 there exists a unique distinguished triangle in $\DAg S$:
$$
j_\sharp j^*(K) \xrightarrow{j_\star}
 K  \xrightarrow{i^\star} i_*i^*(K) \rightarrow j_\sharp j^*(K)[1]
$$
where $j_\star$ (resp. $i^\star$) is the counit (resp. unit)
 of the adjunction $(j_\sharp,j^*)$ (resp. $(i^*,i_*)$).

As we have also seen in Remark \ref{rem:generators}
 that $\DAg S$ is compactly generated,
 we can apply to it the cross-functor theorem of Ayoub and Voevodsky
 (cf. \cite[2.4.50]{CD3}) which we state here for future reference.
\end{num}

\begin{thm}\label{thm:6functors}
Consider the above notations. Then, for any separated morphism of finite type $f:Y \rightarrow X$ of schemes,
 there exists a pair of adjoint functors, the \emph{exceptional functors},
$$
f_!:\DAg Y \rightleftarrows \DAg X:f^!
$$
such that:
\begin{enumerate}
\item There exists a structure of a covariant (resp. contravariant) 
 $2$-functor on $f \mapsto f_!$ (resp. $f \mapsto f^!$).
\item There exists a natural transformation $\alpha_f:f_! \rightarrow f_*$
 which is an isomorphism when $f$ is proper.
 Moreover, $\alpha$ is a morphism of $2$-functors.
\item For any smooth separated morphism of finite type $f:X \rightarrow S$
 of schemes with tangent bundle $T_f$,
 there are canonical natural isomorphisms
\begin{align*}
\piso_f:f_\sharp & \longrightarrow f_!\big(Th_X(T_f) \otimes_X .\big) \\
\piso'_f:f^* & \longrightarrow Th_X(-T_f) \otimes_X f^!
\end{align*}
which are dual to each other
 -- the Thom premotive $Th_X(T_f)$ is $\otimes$-invertible
 with inverse $Th_X(-T_f)$.
\item For any cartesian square:
$$
\xymatrix@=16pt{
Y'\ar^{f'}[r]\ar_{g'}[d]\ar@{}|\Delta[rd] & X'\ar^g[d] \\
Y\ar_f[r] & X,
}
$$
such that $f$ is separated of finite type,
there exist natural isomorphisms
\begin{align*}
g^*f_! \xrightarrow\sim f'_!{g'}^*\, , \\
g'_*{f'}^! \xrightarrow\sim  f^!g_*\, .
\end{align*}
\item For any separated morphism of finite type $f:Y \rightarrow X$ and any Tate spectra $K$ and $L$,
 there exist natural isomorphisms
\begin{align*}
Ex(f_!^*,\otimes):
(f_!K) \otimes_X L &\xrightarrow{\ \sim\ } f_!(K \otimes_Y f^*L)\, ,\ \\
  \uHom_X(f_!(L),K) & \xrightarrow{\ \sim\ } f_* \uHom_Y(L,f^!(K))\, ,\ \\
  f^! \uHom_X(L,M)& \xrightarrow{\ \sim\ } \uHom_Y(f^*(L),f^!(M))\, .
\end{align*}
\end{enumerate}
\end{thm}

\begin{num}\label{num:basic_6functors}
This theorem has many important applications. Let us state
 a few consequences for future use.
\begin{itemize}
\item \textit{Localization triangles}.-- 
Consider again the assumptions of Paragraph \ref{num:localization}.
 Then one gets canonical distinguished triangles:
\begin{align}
\label{eq:localization1}
j_! j^!(K) \xrightarrow{j_\star}
 K  \xrightarrow{i^\star} i_*i^*(K) \rightarrow j_! j^!(K)[1] \\
\label{eq:localization2}
i_! i^!(K) \xrightarrow{i_\star}
 K  \xrightarrow{j^\star} j_*j^*(K) \rightarrow i_! i^!(K)[1]
\end{align}
where $j_\star$,  $j^\star$,  $i_\star$,  $i^\star$ are the unit/counit
 morphism of the obvious adjunctions.
\item \textit{Descent properties}.-- Consider a cartesian square
 of schemes:
$$
\xymatrix@=10pt{
W\ar^k[r]\ar_g[d]\ar@{}|\Delta[rd] & V\ar^f[d] \\
Y\ar_i[r] & X
}
$$
One says $\Delta$ is Nisnevich (resp. cdh) distinguished
 if $i$ is an open (resp. closed) immersion,
 $f$ is an \'etale (resp. proper) morphism and the induced
 map $(V-W) \rightarrow (X-Y)$ of underlying reduced subschemes
 is an isomorphism.

If $\Delta$ is Nisnevich or cdh distinguished,
 then for any object $K \in \DAg X$,
 there exists canonical distinguished triangles:
\begin{align}
\label{eq:descent1}
&K \xrightarrow{i^\star+f^\star} i_*i^*(K) \oplus f_*f^*(K)
 \xrightarrow{k^\star-g^\star} h_*h^*(K) \rightarrow K[1] \\
\label{eq:descent2}
&h_!h^!(K) \xrightarrow{k_\star-g_\star} i_!i^!(K) \oplus f_!f^!(K)
 \xrightarrow{i_\star+f_\star} K \rightarrow h_!h^!(K)[1]
\end{align}
where $f^\star$ (resp. $f_\star$) is the unit (resp. counit)
 of the adjunction $(f^*,f_*)$ (resp. $(f_!,f^!)$) and $h=ig=fk$.
\item \textit{Pairing}.-- Let us apply point (5) replacing $K$
 by $f^!(K)$;
 one gets an isomorphism which appears in the following
 composite map:
$$
f_!\big(f^!(K) \otimes f^*(L) \big) \xrightarrow{\ \sim\ }
 f_!f^!(K) \otimes L \xrightarrow{f_\star \otimes Id_L} K \otimes L,
$$
where $f_\star$ is the counit map for the the adjunction $(f_!,f^!)$.
 Thus by adjunction, one gets a canonical morphism:
$$
f^!(K) \otimes f^*(L) \rightarrow f^!(K \otimes L).
$$
We will see in Paragraph \ref{num:products} that this pairing induces
 the classical cap-product.
\end{itemize}
\end{num}

\begin{rem}\label{rem:chg_coeff}
Let $R$ be a ring of coefficients.
 One can obviously extend the above considerations 
 by replacing sheaves of abelian groups by sheaves of $R$-modules
 (as in \cite{CD3}).
We get a triangulated $R$-linear category $\DAg{S,R}$ depending
 on an arbitrary scheme $S$, and also obtain the six functors
 formalism described above. In brief,
 there is no difference between working with $\ZZ$-linear coefficients
 or $R$-linear coefficients.

Besides, one gets an adjunction of additive categories:
$$
\rho_R^*:\shg(S) \leftrightarrows \shg(S,R):\rho^R_*
$$
where $\rho^R_*$ is the functor that forgets the $R$-linear structure.
 The functor $\rho_R^*$ is obtained by taking the associated sheaf
 of the presheaf obtained after applying the extension of scalars functor
 for $R/\ZZ$. Note that the functor $\rho_R^*$ is monoidal.
 According to \cite[5.3.36]{CD3}, these adjoint functors can be derived
 and further induce adjunctions of triangulated categories:
\begin{equation}\label{eq:chg_coeff}
\derL \rho_R^*:\DAg S \leftrightarrows \DAg{S,R}:\derR \rho^R_*
\end{equation}
such that $\derL \rho^*_R$ is monoidal. 
\end{rem}

\subsection{Ring spectra}

Let us start with a very classical definition.
\begin{df}
Let $S$ be a base scheme.
 A ring spectrum $\E$ over $S$ is a commutative monoid 
 of the monoidal category $\DAg S$.

A morphism of ring spectra is a morphism of commutative monoids.
\end{df}
In other words, $\E$ is a Tate spectrum equipped with
 a multiplication (resp. unit) map
$$
\mu:\E \otimes \E \rightarrow \E,
 \text{ resp. } \eta:\un_S \rightarrow \E
$$
such that the following diagrams are commutative
\begin{equation}\label{eq:axiom_ring}
\xymatrix@R=10pt{
\ar@{}|{\text{Unity:}}[r] &&\ar@{}|{\text{Associativity:}}[r]
&&\ar@{}|{\text{Commutativity:}}[r] & \\
\E\ar^-{1 \otimes \eta}[r]\ar@{=}[rdd] & \E \otimes \E\ar^\mu[dd]
 & \E \otimes \E \otimes \E\ar^-{1 \otimes \mu}[r]\ar_{\mu \otimes 1}[dd]
  &  \E \otimes \E\ar^\mu[dd]
 &  \E \otimes \E\ar^\mu[rd]\ar_\gamma^\sim[dd] & \\
& & & & & \E \\
& \E
 & \E \otimes \E\ar^-\mu[r] & \E
 & \E \otimes \E\ar_\mu[ru]
}
\end{equation}
where $\gamma$ is the isomorphism exchanging factors (coming
 from the underlying symmetric structure of the monoidal category
 $\DAg S$).

\begin{ex}\label{ex:trivial_ring}
\begin{enumerate}
\item The constant Tate spectrum $\un_S$ over $S$ is an obvious example of
 ring spectrum over $S$. Besides, for any ring spectrum $(\E,\mu,\eta)$
 over $S$, the unit map $\eta:\un_S \rightarrow \E$ is a morphism
 of ring spectra.
\item Let $k$ be a fixed base (a perfect field in our main example).
 Let $\E$ be a ring spectrum over $k$.
 Then for any $k$-scheme $S$, with structural morphism $f$, we get a canonical
 ring spectrum structure on $f^*(\E)$ as the functor $f^*$ is monoidal.

In this situation, we will usually denote by $\E_S$ this ring
 spectrum. The family of ring spectra $(\E_S)$ thus defined
 is a cartesian section of the fibered category $\DAg -$.
 It forms what we call an absolute ring spectrum over the category of
 $k$-schemes in \cite{Deg16}.
\end{enumerate}
\end{ex}

\begin{num}\label{num:weak_monoidal}
We finally present a classical recipe in motivic homotopy theory to
 produce ring spectra. Let us fix a base scheme $S$.

Suppose given a triangulated monoidal category $\T$ and
 an adjunction of triangulated categories
$$
 \phi^*:\DAg S \leftrightarrows \T:\phi_*
$$
such that $\phi^*$ is monoidal.

Then, for a couple of objects $K$ and $L$ of $\T$,
 we get a canonical map:
$$
\nu_{\K,L}:
\phi_*(K) \otimes \phi_*(L)
 \xrightarrow{\phi_\star} \phi_*\phi^*\big(\phi_*(K) \otimes \phi_*(L)\big)
 \xrightarrow{\sim} \phi_*\big(\phi^*\phi_*(K) \otimes \phi^*\phi_*(L)\big)
 \xrightarrow{\phi^\star \otimes \phi^\star}
  \phi_*(K \otimes L)
$$
where $\phi_\star$ and $\phi^\star$ are respectively the unit and
 counit of the adjunction $(\phi^*,\phi_*)$.
 Besides, one easily checks that this isomorphism is compatible
 with the associativity and symmetry isomorphisms of $\DAg S$ and $\T$.

We also get a canonical natural transformation
$$
\nu:
\un_S \xrightarrow{\phi^\star}
 \phi_*\phi^*(\un_S) \simeq \phi_*(\un^\T)
$$
which one can check to be compatible with the unit isomorphism of the
 monoidal structures underlying $\DAg S$ and $\T$.
 In other words, the functor $\phi_*$ is weakly monoidal.\footnote{Other
 possible terminologies are weak monoidal functor, lax monoidal functor.}

Then, given any commutative monoid $(M,\mu^M,\eta^M)$ in $\T$,
 one gets after applying the functor $\phi_*$ a ring spectrum with
 multiplication and unit maps
\begin{align*}
\mu:&\phi_*(M) \otimes \phi_*(M) \xrightarrow{\nu_{M,M}}
 \phi_*(M \otimes M) \xrightarrow{\phi_*(\mu^M)} \phi_*(M), \\
\eta:&\un_S \xrightarrow{\nu} \phi_*(\un^\T)
 \xrightarrow{\phi_*(\eta^M)} \phi_*(M).
\end{align*}
The verification of the axioms of a ring spectrum comes from the
 fact that $\phi_*$ is weakly monoidal.
\end{num}

\begin{ex}\label{ex:main_ringsp}
The main example we have in mind is the case where $M$
 is the unit object $\un^\T$ of the monoidal category $\T$:
$$
\E=\phi_*(\un^\T).
$$
\end{ex}

\begin{rem}
On can also define a strict ring spectrum over a scheme $S$
 as a commutative monoid object of the underlying model category
 of $\DAg S$ --- in other words a Tate spectrum equipped with a ring structure
 such that the diagrams \eqref{eq:axiom_ring} commute in the category
 of Tate spectra, rather than in its localization with respect
 to weak equivalences.

In the subsequent cases where ring spectra will appear in this paper,
 through an adjunction $(\phi^*,\phi_*)$ as above,
 this adjunction will be derived from a Quillen adjunction of
 monoidal model categories. We can repeat the arguments above by
 replacing the categories with their underlying model categories.
 Therefore, the ring spectra of the form $\phi_*(\un^\T)$ will in
 fact be strict ring spectra. While we will not use this fact here,
 it is an information that could be useful to the reader.
\end{rem}

\begin{ex}
Let $R$ be a ring of coefficients.
 Consider the notations of Remark \ref{rem:chg_coeff}.
 Then we can apply the preceding considerations to the adjunction
 \eqref{eq:chg_coeff} so that we get a ring spectrum
$$
\HH_{\AA^1}R_X:=\derR \rho^R_*(\un_X).
$$
\end{ex}

\begin{num}\label{num:morphisms}
Consider again the notations of the paragraph preceding the remark.
As mentioned in Example \ref{ex:trivial_ring}(1),
 the ring spectrum $\phi_*(\un^\T)$ automatically comes with
 a morphism of ring spectra:
\begin{equation}\label{eq:regulator}
\un_S \rightarrow \phi_*(\un^\T).
\end{equation}
Moreover, suppose there exists another triangulated monoidal
 categories $\T'$ with an adjunction of triangulated categories:
$$
\psi^*:\T \leftrightarrows \T':\psi_*
$$
such that $\psi^*$ is monoidal.

Then one gets a canonical morphism of ring spectra:
$$
\phi_*(\un^\T)
 \rightarrow \phi_*\psi_* \psi^*(\un^\T) \simeq \phi_*\psi_*(\un^{\T'})
$$
which is compatible with the canonical morphism of
 the form \eqref{eq:regulator}.
\end{num}

%% file: theories.tex
\begin{num}\label{num:notations_4theories}
In this section, we will fix a base scheme $k$
 together with a ring spectrum $(\E,\mu,\eta)$ over $k$.
 Given any $k$-scheme $X$ with structural morphism $f$,
 we denote by $\E_X=f^*(\E)$ the pullback ring spectrum over $X$
 (Example \ref{ex:trivial_ring}).
 We will still denote by $\mu$ (resp. $\eta$) the multiplication
 (resp. unit) map of the ring spectrum $\E_X$.

When $X/k$ is separated of finite type,
 it will also be useful to introduce the following notation:
$$
\E'_X=f^!(\E).
$$
Note that this is not a ring spectrum in general, but that there is a pairing 
\begin{equation}\label{eq:pre_capp}
\mu':\E'_X \otimes \E_X=f^!(\E) \otimes f^*(\E)
 \rightarrow f^!(\E \otimes \E) \xrightarrow \mu f^!(\E)=\E'_X
\end{equation}
using the last point
 of  Paragraph \ref{num:basic_6functors}.

In the following subsections, we will show how to associate
 cohomological/homological theories with $\E$ and deduce
 the rich formalism derived from the six functors formalism
 (Theorem \ref{thm:6functors}).

Note finally that the constructions will be functorial in the
 ring spectra $E$. To illustrate this fact, we will also fix
 a generic morphism of ring spectra
$$
\phi:\E \rightarrow \F.
$$
\end{num}

\subsection{Definitions and basic properties}

\begin{df}\label{df:4theories}
Let $p:X \rightarrow \spec(k)$ be a $k$-scheme,
 $n \in \ZZ$ an integer and $v \in \cK(X)$ a virtual vector bundle over $X$.

We define the cohomology of $X$ in degree $(n,v)$
 and coefficients in $\E$ as the following abelian group:
$$
\E^{n}(X,v):=
\Hom_{\DAg k}\Big(\un_k,p_*\big(p^*(\E) \otimes \Th(v)\big)[n]\Big).
$$
If $X$ is a $k$-variety,
 we also define respectively the cohomology with compact
 support, Borel-Moore homology and homology
 of $X$ in degree $(n,v)$ and coefficients in $\E$ as:
\begin{align*}
\E^n_c(X,v)&:=
\Hom_{\DAg k}\Big(\un_k,p_!\big(p^*(\E) \otimes \Th(v)\big)[n]\Big), \\
\E_{n}^{BM}(X,v)&:=
\Hom_{\DAg k}\Big(\un_k,p_*\big(p^!(\E) \otimes \Th(-v)\big)[-n]\Big), \\
\E_{n}(X,v)&:=
\Hom_{\DAg k}\Big(\un_k,p_!\big(p^!(\E) \otimes \Th(-v)\big)[-n]\Big). \\
\end{align*}
We will sometime use the abbreviations
 \emph{c-cohomology} and \emph{BM-homology} for cohomology with compact support
 and Borel-Moore homology respectively.

Finally, when one replaces $v$ by an integer $m \in \ZZ$,
 we will mean that $v$ is the (opposite of the) trivial vector bundle of
 rank $\vert m\vert$.
\end{df}
We will describe below the properties satisfied by these four theories,
 which can be seen as a generalization of the classical 
 Bloch-Ogus formalism (see \cite{BO}).

\begin{rem}\label{rem:twists&K0}
It is clear from the construction of Paragraph \ref{num:basic_Thom}
 and from the above definition that cohomology and cohomology with compact
 support (resp. Borel-Moore homology and homology) depends covariantly
 (resp. contravariantly) upon the virtual vector bundle $v$
 --- \emph{i.e.} with respect to morphism of the category
 $\mathcal K(X)$.
 In particular, if one consider these theories up to isomorphism,
 one can take for $v$ a class in the $K$-theory ring $K_0(X)$
 of vector bundles over $X$.
\end{rem}

\begin{ex}\label{ex:coh_smooth}
Let us assume that $X$ is a smooth $k$-scheme, with structural
 morphism $p$. Consider a couple of integers $(n,m) \in \ZZ^2$.

Then, one gets the following computations:
\begin{align*}
E^n(X,m)
 &=\Hom_{\DAg k}\big(\un_k,p_*\big(p^*(\E) \otimes \Th(\AA^m_X)\big)[n]\big) \\
 &=\Hom_{\DAg k}\big(\un_k,p_*p^*\big(\E \otimes \Th(\AA^m_k)\big)[n]\big) \\
 &\stackrel{(1)}=\Hom_{\DAg k}\big(p_\sharp p^*(\un_k),\E \otimes \Th(\AA^m_k)[n]\big) \\
 &\stackrel{(2)}=\Hom_{\DAg k}\big(\ZZ_k(X),\E(m)[n+2m]\big) \\
 &=\E^{n+2m,m}(X). \\
\end{align*}
The identification (1) comes from the (derived) adjunctions
 described in Paragraph \ref{num:basic_funct}
 and (2) comes from the definition of $p^*$ (resp. $p_\sharp$)
 --- see again Paragraph \ref{num:basic_funct}.

So for smooth $k$-schemes and constant virtual vector bundles,
 the cohomology theory just defined agree (up to change of twists)
 with the classical cohomology represented by $\E$.
\end{ex}

\begin{rem}\label{rem:relative_MWspectrum}
Using the conventions stated in the beginning of this section,
 one can rewrite the previous definitions as follows:
\begin{align*}
\E^{n}(X,v)
 & =\Hom_{\DAg X}\big(\un_X,\E_X \otimes \Th(v)[n]\big), \\
\E^{n}_{c}(X,v)
 & =\Hom_{\DAg k}\big(\un_k,p_!(\E_X \otimes \Th(v))[n])\big), \\
\E_{n}^{BM}(X,v)
 & =\Hom_{\DAg X}\big(\un_X,\E'_X \otimes \Th(-v)[-n]\big), \\
\E_{n}(X,v)
 & =\Hom_{\DAg k}\big(\un_k,p_!(\E'_X \otimes \Th(-v))[-n]\big).
\end{align*}
In particular, if one interpret $\E_X'$ as a \emph{dual} of $\E_X$,
 our definition of Borel-Moore homology is analogue
 to that of Borel and Moore relative to singular homology
 (see \cite{BM}).
\end{rem}

\begin{num}
Assume $p:X \rightarrow \spec(k)$ is separated of finite type.
 From the natural transformation $\alpha_p:p_! \rightarrow p_*$
 of Theorem \ref{thm:6functors}(2),
 one gets canonical natural transformations:
\begin{align*}
\E^{n}_{c}(X,v) & \rightarrow \E^{n}(X,v), \\
\E_{n}(X,v) & \rightarrow \E_{n}^{BM}(X,v)
\end{align*}
which are isomorphisms whenever $X/k$ is proper.
\end{num}

\begin{rem}\label{rem:pre_homotopy}
Consider an arbitrary $k$-scheme $X$.

One must be careful about the homotopy invariance property.
 Indeed, if $v$ is a virtual bundle over $\AA^1_X$ which comes from $X$,
 that is $v=\pi^{-1}(v_0)$ where $\pi:\AA^1_X \rightarrow X$ is the
 canonical projection, then one gets:
$$
\E^n(\AA^1_X,v) \simeq \E^n(X,v_0)
$$
from the homotopy property of $\DAg X$ --- more precisely,
 the isomorphism $Id \xrightarrow \sim \pi_*\pi^*$.

This will always happen if $X$ is regular.
 But in general, $v$ could not be of the form $\pi^{-1}(v_0)$
 and there is no formula as above.

Similarly, if $v=p^{-1}(v_0)$, one gets:
$$
\E_n(\AA^1_X,v) \simeq \E_n(X,v_0).
$$
Note finally there is no such formula for c-cohomology or
 BM-homology.
\end{rem}

\begin{num}\label{num:morphisms_ringsp}
It is clear that the morphism of ring spectra $\phi:\E \rightarrow \F$
 induces morphisms of abelian groups, all denoted by $\phi_*$:
\begin{align*}
\E^{n}(X,v) &\rightarrow \F^{n}(X,v) \\
\E^{n}_c(X,v) &\rightarrow \F^{n}_c(X,v) \\
\E_{n}^{BM}(X,v) &\rightarrow \F_{n}^{BM}(X,v) \\
\E_{n}(X,v) &\rightarrow \F_{n}(X,v).
\end{align*}
\end{num}

\subsection{Functoriality properties} \label{sec:functoriality}

\begin{num} \label{num:basic_functoriality}
\textit{Basic functoriality}.
Let $f:Y \rightarrow X$ be a morphism of $k$-schemes
 and consider $(n,v) \in \ZZ \times \cK(X)$.

Letting $p$ (resp. $q$) be the projection of $X/k$ (resp. $Y/k$),
 we deduce the following maps in $\DAg X$, where in the second one
 we have assume that $p$ and $q$ are separated of finite type
\begin{align*}
\Th(v) \otimes p^*(\E)
 \xrightarrow{f^\star} f_*f^*(\Th(v) \otimes p^*\E)
 \stackrel{(1)} \simeq f_*\big(\Th(f^{-1}v) \otimes q^*\E\big) \\
f_!\big(\Th(f^{-1}v) \otimes q^!\E\big)
 \stackrel{(2)}\simeq f_!f^!(\Th(v) \otimes p^!\E)
 \xrightarrow{f_\star} \Th(v) \otimes p^!(\E)
\end{align*}
where $f^\star$ (resp. $f_\star$) is the unit (resp. counit) map
 of the adjunction $(f^*,f_*)$ (resp. $(f_!,f^!)$)
 and the isomorphism (1) (resp. (2)) follows from the fact $f^*$ is monoidal
 (resp. $\Th(v)$ is $\otimes$-invertible).

Composing respectively with $p_*$ and $p_!$, we get canonical morphisms:
\begin{align*}
p_*\big(\Th(v) \otimes p^*\E\big)
 &\xrightarrow{\ \pi(f)\ } q_*\big(\Th(f^{-1}v) \otimes q^*\E\big) \\
q_!\big(\Th(f^{-1}v) \otimes q^!\E\big)
 &\xrightarrow{\ \pi'(f)\ }  p_!\big(\Th(v) \otimes p^!\E\big),
\end{align*}
which induces the following pullback and pushforward maps:
\begin{align*}
\E^{n}(X,v) &\xrightarrow{\ f^*\ } \E^{n}(Y,f^{-1}v) \\
\E_{n}(Y,f^{-1}v) &\xrightarrow{\ f_*\ }  \E_{n}(X,v).
\end{align*}
It is straightforward to check that these maps are compatible with composition,
 turning cohomology (resp. homology) into a \emph{contravariant}
 (resp. \emph{covariant})
 functor with respect to all $k$-schemes 
 (resp. all $k$-varieties).

Assume now that $f$ is proper.
Then from Therorem \ref{thm:6functors}(2),
 on gets a canonical isomorphism $\alpha_f:f_! \simeq f_*$
 and the map $\pi(f)$, $\pi'(f)$ respectively induces canonical morphisms:
\begin{align*}
\E^{n}_{c}(X,v) &\xrightarrow{\ f^*\ } \E^{n}_c(Y,f^{-1}v) \\
\E_{n}^{BM}(Y,f^{-1}v) &\xrightarrow{\ f_*\ }  \E_{n}^{BM}(X,v).
\end{align*}
Again, notably because $\alpha_f$ is compatible with composition,
 these maps are compatible with composition so
 that cohomology with compact support
 (resp. Borel-Moore homology) is a \emph{contravariant}
 (resp. \emph{covariant})
 functor with respect to \emph{proper morphisms} of $k$-varieties.
\end{num}

\begin{rem}
With that functoriality at our disposal, we can understand
 the homotopy property described in Remark \ref{rem:pre_homotopy}
 as follows. Given any scheme $X$ and any virtual bundle $v_0$ over $X$,
 the canonical projection $\pi:\AA^1_X \rightarrow X$ induces
 isomorphisms:
\begin{align*}
\pi^*:&\E^{n}(X,v_0) \rightarrow \E^{n}(\AA^1_X,\pi^{-1}(v_0)) ,\\
\pi_*:&\E_{n}(\AA^1_X,\pi^{-1}(v_0)) \rightarrow \E_{n}(X,v_0).
\end{align*}
\end{rem}

\begin{num}\label{num:localization_BM&c}
\textit{Localization long exact sequences}.
 One of the main
 properties of Borel-Moore homolo\-gy, as well as cohomology with
 compact support is the existence of the so-called localization long exact
 sequences.
 In our case, it follows directly from the localization triangle
 stated in Paragraph \ref{num:basic_6functors}.

Indeed for a closed immersion $i:Z \rightarrow X$ of $k$-varieties
 with complement open immersion $j:U \rightarrow X$,
 and a virtual vector bundle $v$ over $X$,
 one gets localization sequences:
\begin{align*}
\E^{BM}_n(Z,i^{-1}v)
& \xrightarrow{i_*} \E^{BM}_n(X,v)
 \xrightarrow{j^*} \E^{BM}_n(U,j^{-1}v)
 \rightarrow \E^{BM}_{n-1}(Z,i^{-1}v), \\
\E_c^n(U,j^{-1}v)
& \xrightarrow{j_*} \E_c^n(X,v)
 \xrightarrow{i^*} \E_c^n(Z,i^{-1}v)
 \rightarrow \E_c^{n+1}(U,j^{-1}v).
\end{align*}
More explicitly, the first (resp. second) exact sequence is obtained by
 using the distinguished triangle \eqref{eq:localization2} with $K=\E'_X$
 (resp. \eqref{eq:localization1} with $K=\E_X$)
 and applying the cohomological functor $\Hom_{\DAg X}(\un_X,-)$.
 Note we also use the identifications $i_!=i_*$
 (resp. $j^!=j^*$) which follows from Theorem \ref{thm:6functors}
 point (2) (resp. (3)).
\end{num}

\begin{num}\label{num:Gysin}
\textit{Gysin morphisms}.
Let us fix a morphism $f:Y \rightarrow X$ of $k$-schemes which is separated
 of finite type and consider the notations
 of Remark \ref{rem:relative_MWspectrum}.

Assume $f$ is smooth with tangent bundle $\tau_f$.
 Then, according to Theorem \ref{thm:6functors}(3) and
 the $\otimes$-invertibility of Thom spectra,
 we get a canonical isomorphism:
$$
\piso'_f:f^!(\E_X) \simeq f^*(\E_X) \otimes \Th(\tau_f)
 =\E_Y \otimes \Th(\tau_f).
$$
Suppose now that $X$ and $Y$ are smooth $k$-varieties,
 with respective structural morphisms $p$ and $q$.
 Then $f$ is a local complete intersection morphism, and has for
 relative virtual tangent bundle the virtual bundle in $\cK(Y)$:
$$
\tau_f=[T_q]-[f^{-1}(T_p)].
$$
Then one can compute $q^!\E$ in two ways:
\begin{align*}
q^!\E & \stackrel{(1)}\simeq q^*(\E) \otimes \Th(T_q)
 =\E_Y \otimes \Th(T_q) \\
& =f^!p^!(\E) \stackrel{(2)}\simeq f^! \big( p^*(\E) \otimes \Th(T_p) \big)
 \stackrel{(3)} \simeq f^!(\E_X) \otimes \Th(f^{-1}T_p)
\end{align*}
where (1) and (2) are given by the relative purity isomorphisms
 of Theorem \ref{thm:6functors}(3),
 respectively for $p$ and $q$,
 and (3) follows from the fact $\Th(T_p)$ is $\otimes$-invertible.
 Putting the two formulas together, one gets as in the previous case
 an isomorphism:
\begin{equation}\label{eq:fdl_class1}
\tilde \eta_f:f^!(\E_X) \simeq \E_Y \otimes \Th(\tau_f).
\end{equation}
Similarly, using the same procedure but exchanging the role of $f^*$
 and $f^!$, one gets a canonical isomorphism:
\begin{equation}\label{eq:fdl_class2}
\tilde \eta'_f:f^*(\E'_X) \simeq \E'_Y \otimes \Th(-\tau_f),
\end{equation}
assuming either $f$ is smooth or $f$ is a morphism of smooth $k$-varieties.

Therefore one gets using adjunctions the following trace maps:
\begin{align*}
tr_f&:f_!\big(\E_Y \otimes \Th(\tau_f)\big) \longrightarrow \E_X, \\
tr'_f&:\E'_X \longrightarrow f_*\big(\E'_Y \otimes \Th(-\tau_f)\big).
\end{align*}
We can tensor these maps with the Thom space of an arbitrary
 virtual vector bundle $v$ over $X$, and compose the map with $p_!$
 for the first one and $p_*$ for the second one to get
 the following maps:
\begin{align*}
q_!\big(q^*\E \otimes \Th(f^{-1}v+\tau_f)\big)
 &\longrightarrow p_!(p^*\E \otimes \Th(v)), \\
p_*(p^!\E_X \otimes \Th(v))
 &\longrightarrow q_*\big(q^!\E \otimes \Th(f^{-1}v-\tau_f)\big).
\end{align*}
If we assume moreover that $f$ is proper,
 then we get using the same procedure and using the identification $f_*=f_!$
 the following maps:
\begin{align*}
q_*\big(q^*\E \otimes \Th(f^{-1}v+\tau_f)\big)
 &\longrightarrow p_*(p^*\E \otimes \Th(v)), \\
p_!(p^!\E_X \otimes \Th(v))
 &\longrightarrow q_!\big(q^!\E \otimes \Th(f^{-1}v-\tau_f)\big).
\end{align*}
Let us state the result in term of the four theories in the following
 proposition.
\end{num}
\begin{prop}\label{prop:Gysin}
Let 
$f:Y \rightarrow X$ be a morphism of $k$-varieties satisfying one of the
 following assumptions:
\begin{enumerate}
\item[(a)] $f$ is smooth;
\item[(b)] $X$ and $Y$ are smooth $k$-varieties.
\end{enumerate}
Then the maps defined above induce the following
 \emph{Gysin morphisms}:
\begin{align*}
f_*:&\E^{n}_{c}(Y,f^{-1}v+\tau_f) \longrightarrow \E^{n}_{c}(X,v), \\
f^*:&\E_{n}^{BM}(X,v) \longrightarrow \E_{n}^{BM}(Y,f^{-1}v-\tau_f).
\end{align*}
Assume moreover that $f$ is proper.
 Then using again the previous constructions,
 one gets the following maps:
\begin{align*}
f_*:&\E^{n}(Y,f^{-1}v+\tau_f) \longrightarrow \E^{n}(X,v), \\
f^*:&\E_{n}(X,v) \longrightarrow \E_{n}(Y,f^{-1}v-\tau_f).
\end{align*}
These Gysin morphisms are compatible with composition.

Under assumption (a), for any cartesian square,
$$
\xymatrix@=10pt{
Y'\ar^g[r]\ar_q[d] & X'\ar^p[d] \\
Y\ar^f[r] & X
}
$$
one has the classical base change formulas:
\begin{itemize}
\item $p^*f_*=g_*q^*$ in case of cohomologies,
\item $f^*p_*=q_*g^*$ in case of homologies.
\end{itemize}
\end{prop}
Indeed, the construction of maps are directly obtained
 from the maps defined in the paragraph preceding the proposition.
 The compatibility with composition is a straightforward check once
 we use the compatibility of the relative purity isomorphism with
 composition (due to Ayoub, see \cite[2.4.52]{CD3} for the precise statement).
 The base change formulas in the smooth case is similar
 and are ultimately reduced to the compatibility of the relative purity
 isomorphism with base change
 (see for example the proof of \cite[Lemma 2.3.13]{Deg16}).

\begin{rem}
Gysin morphisms can be defined under the weaker assumption
 that $f$ is a global complete intersection.
 Similarly, the base change formula can be extended to cover
 also the case of assumption (b), as well as the general case.
 We refer the reader to \cite{DJK} for this generality, as well as for
 more details on the proofs.
\end{rem}

\begin{rem}\label{rem:morphisms_ringsp1}
According to these constructions, 
 it is clear that the map associated in Paragraph
 \ref{num:morphisms_ringsp} with the morphism of ring spectra $\phi$
 are natural in $X$ with respect to the basic functoriality
 and Gysin morphisms of each of the four theories.
\end{rem}

\subsection{Products and duality}

\begin{num}\label{num:products}
As usual, one can define a \emph{cup-product} on cohomology,
$$
\E^n(X,v) \otimes \E^m(X,w) \rightarrow \E^{n+m}(X,v+w),
 (x,y) \mapsto x \cupp y
$$
where, using the presentation of Remark \ref{rem:relative_MWspectrum}
 one defines $x \cupp y$ as the map:
\begin{align*}
\un_X & \xrightarrow{x \otimes y}
 \E_X \otimes \Th(v) \otimes \E_X \otimes \Th(w)[n+m]
 \simeq  \E_X \otimes \E_X \otimes \Th(v+w)[n+m] \\
 & \xrightarrow{\ \mu \otimes Id\ } \E_X \otimes \Th(v+w)[n+m].
\end{align*}
Here we use the fact that $p_*$ is weakly monoidal --- as the right adjoint
 of a monoidal functor; see Paragraph \ref{num:weak_monoidal}.
 One can easily check that the pullback morphism on cohomology
 is compatible with cup product
 (see for example \cite[1.2.10, (E5)]{Deg12}).

Besides one gets a \emph{cap-product}:
$$
\E_n^{BM}(X,v) \otimes \E^m(X,w) \rightarrow \E^{BM}_{n-m}(X,v-w),
 (x,y) \mapsto x \capp y
$$
defined, using again the presentation of Remark \ref{rem:relative_MWspectrum},
 as follows:
\begin{align*}
x \capp y:\un_X[n-m] & \xrightarrow{x \otimes y}
 \E'_X \otimes \Th(-v) \otimes \E_X \otimes \Th(w)
 \simeq  \E'_X \otimes \E_X \otimes \Th(-v+w) \\
 & \xrightarrow{\ \mu'\ } \E'_X \otimes \Th(-v+w).
\end{align*}
where $\mu'$ is defined in \eqref{eq:pre_capp}.

There is finally a \emph{cap-product with support}:
$$
\E_n^{BM}(X,v) \otimes \E^m_c(X,w) \rightarrow \E_{n-m}(X,v-w),
 (x,y) \mapsto x \capp y
$$
defined, using Remark \ref{rem:relative_MWspectrum}, as follows:
\begin{align*}
x \capp y:\un_k[n-m]
 & \xrightarrow{x \otimes y}
  p_*\big(\E'_X \otimes \Th(-v)\big) \otimes p_!(\E_X \otimes \Th(w)) \\
 & \xrightarrow{(1)} 
  p_!\big(p^*p_*(\E'_X \otimes \Th(-v)\big) \otimes \E_X \otimes \Th(w) \big) \\
 & \xrightarrow{ad'} 
  p_!\big(\E'_X \otimes \Th(-v) \otimes \E_X \otimes \Th(w) \big)
  \simeq  p_!\big((\E'_X \otimes \E_X \otimes \Th(-v+w)\big) \\
 & \xrightarrow{\ \mu'\ } p_!\big((\E'_X \otimes \Th(-v+w))\big).
\end{align*}
where (1) is obtained using the base change isomorphism
 of Theorem \ref{thm:6functors}(5) and $\mu'$ is defined
 in \eqref{eq:pre_capp}.
\end{num}

\begin{rem}\label{rem:PF}
These products satisfies projection formulas with respect to Gysin
 morphisms. We let the formulation to the reader (see also \cite{DJK}).)
\end{rem}

\begin{rem}\label{rem:morphisms_ringsp2}
Clearly, the map
$$
\phi_*:\E^n(X,v) \rightarrow \F^n(X,v)
$$
defined in \ref{num:morphisms_ringsp} is compatible with cup-products.
Similarly, the other natural transformations associated with 
 the morphism of ring spectra $\phi$ are compatible with cap-products.
\end{rem}

\begin{num}\label{num:fdl}
\textit{Fundamental class}.
Let us fix a smooth $k$-variety $f:X \rightarrow \spec(k)$,
 with tangent bundle $\tau_X$.

Applying theorem \ref{thm:6functors}(3) to $f$, we get an isomorphism:
\[
\piso'_f:\E_X=f^*(\E) \longrightarrow Th_X(-T_X) \otimes f^!(\E)=Th_X(-T_X) \otimes \E^\prime_X
\]
Then, in view of Remark \ref{rem:relative_MWspectrum},
 the composite map:
\[
\eta_X:\un_X\xrightarrow{\eta} \E_X\xrightarrow{\piso'_f}Th_X(-T_X) \otimes \E^\prime_X\simeq \E^\prime_X\otimes Th_X(-T_X)
\]
corresponds to a class in the Borel-Moore homology
 group $\E_0^{BM}(X,T_X)$.
\end{num}
\begin{df}\label{df:fdl_class}
Under the assumptions above, we call the class
 $\eta_X \in \E_0^{BM}(X,T_X)$ the fundamental class of
 the smooth $k$-scheme $X$ with coefficients in $\E$.
\end{df}

The following Poincar\'e duality theorem is now a mere consequence of
 the definitions and of part (3) of  Theorem \ref{thm:6functors}.
\begin{thm}\label{thm:duality}
Let $X/k$ be a smooth $k$-variety
 and $\eta_X$ its fundamental class with coefficients in $\E$
 as defined above.

Then the following morphisms
\begin{align*}
\E^n(X,v) &\rightarrow \E^{BM}_{-n}(X,T_X-v),
 x \mapsto \eta_X \capp y \\
\E^n_c(X,v) &\rightarrow \E_{-n}(X,T_X-v),
 x \mapsto \eta_X \capp y
\end{align*}
are isomorphisms.
\end{thm}
\begin{proof}
Let us consider the first map. Using Remark \ref{rem:relative_MWspectrum},
 we can rewrite it as follows:
$$
\Hom_{\DAg X}\big(\un_X,\E_X \otimes \Th(v))[n]\big)
 \rightarrow \Hom_{\DAg X}\big(\un_X,\E'_X \otimes \Th(v-T_X))[n]\big).
$$
Then it follows from the definition of the fundamental class
 (Paragraph \ref{num:fdl})
 and that of cap-products (Paragraph \ref{num:products}),
 that this map is induced by the morphism
\[
\piso'_f:\E_X=f^*(\E) \longrightarrow Th_X(-T_X) \otimes f^!(\E)=Th_X(-T_X) \otimes \E^\prime_X
\]
which is an isomorphism according to Theorem \ref{thm:6functors}(3).

The proof in the case of the second isomorphism is the same.
\end{proof}

\begin{rem}
According to this theorem, and the basic functoriality of the four
 theories (Paragraph \ref{num:basic_functoriality}),
 one gets the following exceptional functoriality for morphisms of smooth
 $k$-varieties:
\begin{itemize}
\item cohomology becomes covariant with respect to proper morphisms;
\item BM-homology becomes contravariant with respect to all morphisms;
\item c-cohomology becomes covariant with respect to all morphisms;
\item homology becomes contravariant with respect to proper morphisms.
\end{itemize}
An application of the projection formulas alluded to in Remark
 \ref{rem:PF} give that this extra functorialities coincide
 with Gysin morphisms constructed in Proposition \ref{prop:Gysin}.
\end{rem}

\begin{ex}
One deduces from the above duality isomorphisms and the localization
 long exact sequence of Paragraph \ref{num:localization_BM&c}
 that, for a closed immersion $i:Z \rightarrow X$ of smooth $k$-varieties
 with complement open immersion $j$, one has long exact sequences:
\begin{align*}
\E^n(Z,T_Z-i^{-1}v)
& \xrightarrow{i_*} \E^n(X,T_X-v)
 \xrightarrow{j^*} \E^n(U,j^{-1}(T_X-v))
 \rightarrow \E^{n+1}(Z,T_Z-i^{-1}v), \\
\E_n(U,j^{-1}(T_X-v))
& \xrightarrow{j_*} \E_n(X,T_X-v)
 \xrightarrow{i^*} \E_n(Z,T_Z-i^{-1}v)
 \rightarrow \E_{n+1}(U,j^{-1}(T_X-v)).
\end{align*}
Besides, substituting $(T_X-v)$ to $v$ and using the exact
 sequence over vector bundles over $Z$:
$$
0 \rightarrow T_Z \rightarrow i^{-1}T_X
 \rightarrow N_ZX \rightarrow 0
$$
where $N_ZX$ is the normal bundle of $Z$ in $X$,
 the above exact sequences can be written more simply as:
\begin{align*}
\E^n(Z,i^{-1}v-N_ZX)
& \xrightarrow{i_*} \E^n(X,v)
 \xrightarrow{j^*} \E^n(U,j^{-1}v)
 \rightarrow \E^{n+1}(Z,i^{-1}v-N_ZX), \\
\E_n(U,j^{-1}v)
& \xrightarrow{j_*} \E_n(X,v)
 \xrightarrow{i^*} \E_n(Z,i^{-1}v-N_ZX)
 \rightarrow \E_{n+1}(U,j^{-1}v).
\end{align*}
\end{ex}

\subsection{Descent properties}

The following result is a direct application of Paragraph
 \ref{num:basic_6functors}.
\begin{prop}
Consider a cartesian square of $k$-schemes:
$$
\xymatrix@=10pt{
W\ar^k[r]\ar_g[d]\ar@{}|\Delta[rd] & V\ar^f[d] \\
Y\ar_i[r] & X
}
$$
which is Nisnevich or cdh distinguished
 (see  Paragraph \ref{num:basic_6functors}).
 Put $h=i \circ g$ and let $v$ be a virtual bundle over $X$.

Then one has canonical long exact sequences of the form:
\begin{align*}
\E^n(X,v)
& \xrightarrow{i^*+f^*} \E^n(Y,i^{-1}v) \oplus \E^n(V,f^{-1}v)
 \xrightarrow{k^*-g^*} \E^n(W,h^{-1}v)
 \rightarrow \E^{n+1}(X,v) \\
\E_n(W,h^{-1}v)
& \xrightarrow{k_*-g_*} \E_n(Y,i^{-1}v) \oplus \E_n(V,f^{-1}v)
 \xrightarrow{i_*+f_*} \E_n(X,v)
 \rightarrow \E_{n-1}(X,v).
\end{align*}

If the square $\Delta$ is Nisnevich distinguished
 (in which case all of its morphisms are \'etale),
 one has canonical long exact sequences of the form:
\begin{align*}
\E^{BM}_n(X,v)
& \xrightarrow{i^*+f^*} \E^{BM}_n(Y,i^{-1}v) \oplus \E^{BM}_n(V,f^{-1}v)
 \xrightarrow{k^*-g^*} \E^{BM}_n(W,h^{-1}v)
 \rightarrow \E^{BM}_{n-1}(X,v) \\
\E_c^n(W,h^{-1}v)
& \xrightarrow{k_*-g_*} \E_c^n(Y,i^{-1}v) \oplus \E_c^n(V,f^{-1}v)
 \xrightarrow{i_*+f_*} \E_c^n(X,v)
 \rightarrow \E_c^{n+1}(X,v)
\end{align*}
where we have used the Gysin morphisms with respect to \'etale maps
 for BM-homology and c-cohomology.

If the square $\Delta$ is cdh distinguished (in which case all of its
 morphisms are proper),
 one has canonical long exact sequences of the form:
\begin{align*}
\E^{BM}_n(W,h^{-1}v)
& \xrightarrow{k_*-g_*} \E^{BM}_n(Y,i^{-1}v) \oplus \E^{BM}_n(V,f^{-1}v)
 \xrightarrow{i_*+f_*} \E^{BM}_n(X,v)
 \rightarrow \E^{BM}_{n-1}(X,v), \\
\E^n_c(X,v)
& \xrightarrow{i^*+f^*} \E^n_c(Y,i^{-1}v) \oplus \E^n_c(V,f^{-1}v)
 \xrightarrow{k^*-g^*} \E^n_c(W,h^{-1}v)
 \rightarrow \E^{n+1}_c(X,v)
\end{align*}
where we have used the covariance (resp. contravariance)
 of BM-homology (resp. c-cohomology) constructed
 in Paragraph \ref{num:basic_functoriality}.
\end{prop}
\begin{proof}
The proof is a simple application of the descent properties
 obtained in Paragraph \ref{num:basic_6functors}.

For example, one gets the case of cohomology by using
 the distinguished triangles \eqref{eq:descent1} with $K=\E_X$:
$$
\E_X \xrightarrow{i^\star+f^\star} i_*(\E_Y) \oplus f_*(\E_V)
 \xrightarrow{k^\star-g^\star} h_*(\E_W) \rightarrow \E_X[1]
$$
and applying the cohomological functor $\Hom_{\DAg X}(\un_X,-)$.
 The description of the maps in the long exact sequence obtained
 follows directly from the description of the contravariant
 functoriality of cohomology (see \ref{num:basic_functoriality}).

The other exact sequences are obtained similarly.
\end{proof}

We can state the existence of the long exact sequences in the
 preceding proposition by saying that the four theories satisfies
 Nisnevich and cdh cohomological descent.\footnote{One can express
 cohomological descent for the cdh or Nisnevich topology in the style
 of \cite[Vbis]{SGA4}, or \cite[\textsection 3]{CD3},
 using the fact the four theories admits an extension to simplicial
 schemes and stating that cdh or Nisnevich hypercovers induces
 isomorphisms. The simplification of our formulation comes as
 these topologies are defined by cd-structure in the sense
 of \cite{Voecd}.} This also shows that, 
 under the existence of resolution of singularities,
 they are essentially determined by their restriction to smooth $k$-varieties.

Let us make a precise statement.
\begin{prop}\label{prop:unique_extension}
Let us assume $k$ is of characteristic $0$
 or more generally 
 that any $k$-variety admits a non singular blow-up
 and that $k$ is perfect.

Suppose one has a contravariant functor
 $H^*$ from $k$-varieties equipped with a virtual vector bundle
 $(X,v)$ to graded abelian groups, $H^n(X,v)$,
 and a natural transformation:
$$
\phi_X:\E^n(X,v) \rightarrow H^n(X,v)
$$
such that for any cdh distinguished square as in the preceding proposition,
 one has a long exact sequence:
$$
H^n(X,v)
 \xrightarrow{i^*+f^*} H^n(Y,i^{-1}v) \oplus H^n(V,f^{-1}v)
 \xrightarrow{k^*-g^*} H^n(W,h^{-1}v)
 \rightarrow H^{n+1}(X,v)
$$
which is compatible via $\phi$ with the one for $\E^*$.

Then, if $\phi_X$ is an isomorphism when $X/k$ is smooth,
 it is an isomorphism for any $k$-variety $X$.
\end{prop}
\begin{proof}
The proof is an easy induction on the dimension of $X$.
 When $X$ has dimension $0$, it is necessarily smooth
 over $k$ as the latter field is assumed to be perfect.
 The noetherian induction argument follows from the existence
 of a blow-up $f:V \rightarrow X$ such that $V$ is smooth.
 Let $Y$ be the locus where $f$ is not an isomorphism,
 $W=V \times_X Y$. Then the dimension of $Y$ and $W$ is strictly
 less than the dimension of $X$ and $V$.
 By assumptions, $\phi_V$ is an isomorphism.
 By the inductive assumption, $\phi_Y$ and $\phi_W$ are isomorphisms.
 So the existence of the cdh descent long exact sequences,
 and the fact $\phi$ is compatible with these,
 allow to conclude.
\end{proof}

\begin{rem}
Similar uniqueness statements,
 with the same proof, hold for the other three theories.
 We let the formulation to the reader.
\end{rem}

%% file: mwspectrum.tex
\subsection{The ring spectra}

\begin{num}
We will now apply the machinery described in the preceding section
 to $\mathrm{MW}$-motives. Recall from our notations
 that $k$ is now a perfect field and $R$ is a ring of coefficients.

Recall also from \cite[(3.3.6.a)]{DF1}
 that we have adjunctions of triangulated categories:
\begin{equation}\label{eq:diagram_DMs}
\begin{split}
\xymatrix@C=30pt@R=24pt{
\DA\ar@<+2pt>^{\derL \tilde \gamma^*}[r]\ar@<+2pt>^{a}[d]
 & \DMt\ar@<+2pt>^{\derL \pi^*}[r]\ar@<+2pt>^{\tilde a}[d]
     \ar@<+2pt>^{\tilde \gamma_{*}}[l]
 & \DM\ar@<+2pt>^{a^{tr}}[d]
     \ar@<+2pt>^{\pi_{*}}[l] \\
\DAet\ar@<+2pt>^{\derL \tilde \gamma^*_\et}[r]
    \ar@<+2pt>^{\derR \mathcal O}[u]
 & \DMtx{\et}\ar@<+2pt>^{\derL \pi^*_\et}[r]
	  \ar@<+2pt>^{\derR \mathcal O}[u]
		\ar@<+2pt>^{\tilde \gamma_{\et*}}[l]
 & \DMx{\et}.
    \ar@<+2pt>^{\derR \mathcal O}[u]
		\ar@<+2pt>^{\pi_{\et*}}[l]
}
\end{split}
\end{equation}
such that each left adjoint is monoidal. Therefore, one can apply the general procedure of Paragraph
 \ref{num:weak_monoidal} to deduce ring spectra from these
 adjunctions.
\end{num}
\begin{df}\label{df:M&MW_ring_sp}
Consider the above notations.

We define respectively the $R$-linear MW-spectrum, \'etale MW-spectrum,
 motivic Eilenberg-MacLane spectrum and \'etale Eilenberg-MacLane spectrum
 as follows:
\begin{align*}
\Ht R&:=\rho^R_*\tilde \gamma_*(\un), \\
\Htet R&:=\rho^R_* \derR \mathcal O \tilde \gamma_{\et*}(\un), \\
\HM R&:=\rho^R_*  \tilde \gamma_* \pi_*(\un), \\
\HMet R&:=\rho^R_* \derR \mathcal O \gamma_{\et*} \pi_{\et*}(\un).
\end{align*}
\end{df}

\begin{num}\label{num:trivial_comput_coh}
Each of these ring spectrum represents the corresponding cohomology
 on smooth $k$-schemes. This follows by adjunction using Example
 \ref{ex:coh_smooth}. Explicitly, for a smooth $k$-scheme $X$
 and integer $(n,m) \in \ZZ^2$, one gets:
\begin{align*}
\Ht^n(X,m,R)
&=\Hom_{\DAg k}\big(\Sigma^\infty \ZZ_k(X),\rho^R_*\tilde \gamma_*(\un)(m)[n+2m]\big) \\
&=\Hom_{\DMt}\big(\tilde M(X),\un(m)[n+2m]\big), \\
&=\Ht^{n+2m,m}(X,R)
\end{align*}
where the first identification follows from Example \ref{ex:coh_smooth},
 the second by adjunction --- here, $\tilde M(X)$ denotes the $\mathrm{MW}$-motive
 associated with the smooth $k$-scheme $X$, 
 following the notation of \cite{DF1}.
 The last group was introduced in \cite[4.1.1]{DF1}.

Similarly we get:
\begin{align*}
\HM^n(X,m,R)
&=\Hom_{\DM}\big(M(X),\un(m)[n+2m]\big), \\
\Htet^n(X,m,R)
&=\Hom_{\DMtx{\et}}\big(M_\et(X),\un(m)[n+2m]\big), \\
\HMet^n(X,m,R)
&=\Hom_{\DMx{\et}}\big(\tilde M_\et(X),\un(m)[n+2m]\big)
\end{align*}
where $M(X)$ (resp. $M_\et(X)$, $\tilde M_\et(X)$) denotes the motive
 (resp. \'etale motive, \'etale $\mathrm{MW}$-motive)
 associated with the smooth $k$-scheme $X$.
\end{num}

\begin{rem}\label{rem:on_spectra}
\begin{enumerate}
\item The Eilenberg-MacLane motivic ring spectrum defined here,
 apart the fact that we work in $\DAg X$,
 agrees with the one defined by Voevodsky (\cite[6.1]{V3}; see
 also \cite[11.2.17]{CD3}).
\item The use of the functor $\rho^R_*$ in the above definition
 can appear rather artificial. We could have equally worked within
 the $R$-linear $\AA^1$-derived category $\DAb X$.
\item As in algebraic topology, one derive from the Dold-Kan
 equivalence an adjunction of triangulated categories:
$$
N:\SH \leftrightarrows \DAg k:K
$$
see \cite[5.3.35]{CD3}. The functor $N$ is monoidal (\emph{loc. cit.})
 so that using the arguments of Paragraph \ref{num:weak_monoidal},
 the functor $N$ is weakly monoidal.
 Therefore, if one apply $K$ to any ring spectrum of the above 
 definition, one obtain a commutative monoid in the stable homotopy
 category $\SH$. From the point of view of the cohomology theory
 (and actually, the four theories), this does not change anything
 as the corresponding object do represent the same cohomology.
 This is why we have restricted our attention here to the category $\DAg k$.
 Though it does not change anything for our purpose, it is important
 to know that we have also constructed a ring spectrum in $\SH$
 representing $\mathrm{MW}$-motivic cohomology.
\end{enumerate}
\end{rem}

\subsection{Associated cohomology}

\begin{num}\label{num:MW-sp&concrete_coh}
\textit{$\mathrm{MW}$-cohomology}.
Applying Definition \ref{df:4theories},
 we can associate to the preceding ring spectra four
 cohomological/homological theories.

Our main interest is in the MW-motivic cohomology spectrum.
 First considering the associated cohomology theory,
 our definition allows to extend the definition of
 the $\mathrm{MW}$-cohomology theory of \cite[4.1.1]{DF1},
 with its product,
 to the case of possibly singular $k$-schemes.
 In characteristic $0$,
 this extension is the unique one satisfying cdh descent
 (see Proposition \ref{prop:unique_extension}).
 And finally, we have defined Gysin morphisms on cohomology,
 with respect to proper morphisms of smooth $k$-varieties
 (or proper smooth morphisms of arbitrary $k$-schemes); see Proposition
 \ref{prop:Gysin}.

Recall finally the results of \cite{DF1}, assuming that $k$ is infinite
 (in addition to be perfect).
 For a smooth $k$-scheme $X$ and a couple of integers $(n,m) \in \ZZ^2$
 we have the following computations:
$$
\Ht^n(X,m,\ZZ)=\begin{cases}
\wCH^m(X) & \text{if } n=0, \\
\H^n_\zar(X,\sKMW_0) & \text{if } m=0, \\
\H^{n+2m}_\zar(X,\tilde\ZZ(m)) & \text{if } m>0, \\
\H^{n+m}_\zar(X,\Wi) & \text{if } m<0.
\end{cases}
$$
See \cite[Cor. 4.2.6, Prop. 4.1.2]{DF1}.
\end{num}

\begin{ex} \textit{Motivic cohomology (smooth case)}.
Recall the following classical computations.
 For a smooth $k$-variety $X$ and a virtual bundle $v$ over $X$
 of rank $m$, oner has:
$$
\HM^n(X,v,\ZZ)=\begin{cases}
\mathrm{CH}^m(X) & \text{if } n=0, \\
\H^{n+2m}_\zar(X,\ZZ(m)) & \text{if } m \geq 0, \\
0 & \text{if } m<0,
\end{cases}
$$
where $\ZZ(m)$ is Suslin-Voevodsky motivic complex.\footnote{For example,
 $\ZZ(m)=\underline C_*(\ZZ^{tr}\big(\GG^{\wedge,n})\big)[-n]$,
 where $\underline C_*$ is the Suslin (singular) complex functor.
 See \cite[chap. 5]{FSV}.} Note that the computation uses the fact that
 motivic cohomology is an oriented cohomology theory
 (see for example \cite[2.1.4(2)]{Deg12}).
\end{ex}

\begin{num}\label{num:motivic_singular} \label{num:mot_coh}
\textit{Motivic cohomology (singular case)}.
For any $k$-scheme $f:X \rightarrow \spec{k}$,
 following the conventions of Paragraph \ref{num:notations_4theories},
 we consider the following ring spectrum over $X$:
$$
\HMx X:=f^*(\HM).
$$
Apart from the fact that we are working in $\DAg X$ instead of $\mathrm{SH}(X)$,
 this ring spectrum agrees with the ring spectrum defined in \cite[3.8]{CD5}.

Assume further that:
\begin{equation}\label{eq:assumption_DM}
\text{The characteristic exponent of $k$ is invertible in $R$.}
\end{equation}
Then according to \cite{CD5}, one can extend the construction
 of the category $\DM$ to an arbitrary $k$-scheme $X$ and obtain
 a triangulated $R$-linear category
$$
\DMcdh(X,R)
$$
which satisfies the six functors formalism for various $X$,
 as described in Section
 \ref{sec:6functors} (see \cite[5.11]{CD5}).
 Indeed, $\DMcdh(-,R)$ form what we called a motivic triangulated
 category in \cite[2.4.45]{CD3} --- here, the assumption
 \eqref{eq:assumption_DM} is essential.
 Note in particular that we can define, as in Paragraph \ref{num:basic_Thom},
 Thom motives of virtual bundles (see \cite[2.4.15]{CD3}).
 Given a virtual bundle $v$ over $X$, we denote by $M\Th(v)$
 this Thom motive, an object of $\DMcdh(X,R)$.

Moreover, according to \cite[11.2.16]{CD3},
 one has a natural adjunction of triangulated categories:
$$
\gamma^*:\DAg{X,R} \rightarrow \DMcdh(X,R):\gamma_*
$$
extending the adjunction
 $(\derL \pi^*\derL \tilde \gamma^*,\tilde \gamma_*\pi_*)$.
 In fact, the family of adjunctions $(\gamma^*,\gamma_*)$
 for various schemes $X$ form what
 we called a premotivic adjunction in \cite[1.4.6]{CD3}.

As a consequence, $\gamma^*$ is monoidal and commutes with functors
 of the form $f^*$ and $p_!$ while $\gamma_*$ commutes
 with functors of the form $f_*$ and $p^!$ (see \cite[2.4.53]{CD3});
 here $f$ is any morphism of $k$-schemes while $p$ is separated of finite
 type.

Note that by construction, $\gamma^*(\Th(v))=M\Th(v)$.
 As the objects $\Th(v)$ and $M\Th(v)$ are $\otimes$-invertible,
 one deduces that\footnote{The argument goes as follows. Consider the functors:
$$
\mathcal{T}h(v):K \mapsto \Th(v) \otimes K,
 \mathcal{T}h^M(v):K \mapsto M\Th(v) \otimes K.
$$
Then we obtain an isomorphism of functors:
$$
\gamma^* \circ \mathcal{T}h(v) \simeq \mathcal{T}h^M(v) \circ \gamma^*.
$$
Now, as the Thom objects are $\otimes$-invertible,
 the functors $\mathcal{T}h(v)$ and $\mathcal{T}h^M(v)$ are equivalences
 of categories. Their quasi-inverses are respectively:
 $\mathcal{T}h(-v)$ and $\mathcal{T}h^M(-v)$.
 Then from the preceding isomorphism of functors, one deduces an isomorphism
 of the right adjoint functors:
$$
\mathcal{T}h(-v) \circ \gamma_*
 \simeq \gamma_* \circ \mathcal{T}h^M(-v).
$$}
$$
\gamma_*(M\Th(v))=\Th(v).
$$

Finally, applying \cite[Prop. 4.3 and Th. 5.1]{CD5}, we also obtain
 that $\gamma_*$ commutes with $f^*$. As a consequence,\footnote{Note this
 identification is by no means obvious. In fact,
 it answers a conjecture of Voevodsky (cf. \cite[Conj. 17]{VoeOpen})
 in the particular case of the base change map $f:X \rightarrow \spec(k)$.
 See also \cite[3.3, 3.6]{CD5}.}
$$
\HMx X=f^*(\gamma_*(\un_k))\simeq \gamma_*(f^*(\un_k))=\gamma_*(\un_X).
$$
Finally, under assumption \eqref{eq:assumption_DM}, we can do
 the following computation for an arbitrary $k$-scheme $X$
 and a virtual vector bundle $v$ over $X$ of rank $m$:
\begin{align*}
\HM^n(X,v,R)
&=\Hom_{\DAg X}\big(\un_X,\HMx X \otimes \Th(v))[n]\big) \\
&=\Hom_{\DAg X}\big(\un_X,\gamma_*(\un_X) \otimes \Th(v))[n]\big) \\
&=\Hom_{\DAg X}\big(\un_X,\gamma_*(M\Th(v))[n])\big) \\
&=\Hom_{\DMcdh(X,R)}\big(\un_X,M\Th(v)[n])\big) \\
&=\Hom_{\DMcdh(X,R)}\big(\un_X,\un_X(m)[n+2m])\big) \\
\end{align*}
which uses the identifications recalled above,
 and for the last one, the fact that motivic cohomology is an oriented
 cohomology theory (equivalently, $\DMcdh$ is an oriented motivic
 triangulated category, see \cite[2.4.38, 2.4.40, 11.3.2]{CD3}

We can give more concrete formulas as follows.
 Assume $X$ is a $k$-variety.
 Recall Voevodsky has defined in \cite[chap. 5, \textsection 4.1]{FSV}
 a motivic complex $\underline C_*(X)$
 in $\DMg^{\mathrm{eff}}(k,\ZZ)$ by considering the Suslin complex of the sheaf
 with transfers $L(X)$ represented by $X$.
 With $R$-coefficients, let us put:
$$
\underline C_*(X)_R:=\underline C_*(X) \otimes^\derL_\ZZ R.
$$
 Then according to \cite[8.4, 8.6]{CD5}, one gets:
$$
\HM^n(X,v,R)=\begin{cases}
\Hom_{\DMg^{\mathrm{eff}}(k,R)}\big(\underline C_*(X)_R,R(m)[n+2m]\big)
 & \text{if } m \geq 0, \\
0 & \text{if } m<0
\end{cases}
$$
where $R(m)$ is the $R$-linear Tate motivic complex:
 $R(m)=\underline C_*(\GG^{\wedge m})_R[-m]$.
 Note also that one can compute the right hand side
 as the following cdh-cohomology group (see \cite[(8.3.1)]{CD3}:
$$
\Hom_{\DMg^{\mathrm{eff}}(k,R)}\big(\underline C_*(X),R(m)[n+2m]\big)
 =\HH^{n+2m}_{\mathrm{cdh}}\big(X,R(m)\big),
$$
where $R(m)$ is seen as a complex of cdh-sheaves on the site
 of $k$-schemes of finite type.
\end{num}

\subsection{Associated Borel-Moore homology}

\begin{num}\textit{Borel-Moore motivic $\mathrm{MW}$-homology}.
We also get the Borel-Moore $\mathrm{MW}$-homology
 of $k$-varieties, covariant (resp. contravariant)
 with respect to proper (resp. \'etale) maps,
 satisfying the localization long exact sequence
 (see Paragraph \ref{num:localization_BM&c})
 and contravariant for any smooth maps or arbitrary morphisms
 of smooth $k$-varieties (Gysin morphisms,
 Proposition \ref{prop:Gysin}).

Besides, using the duality theorem \ref{thm:duality},
 we get for a smooth $k$-scheme $X$ with tangent bundle $T_X$
 and any couple of integers $(n,m) \in \ZZ^2$
 the following computation:
$$
\HtBM_n(X,T_X-m,R)=\Ht^{2m-n,m}(X,R)
$$
with the notation of \cite{DF1}.
\end{num}

\begin{num}\textit{Borel-Moore motivic homology}.
\label{num:BM_motivic}
Let us consider the situation of Paragraph \ref{num:motivic_singular}.
 We assume further that $R$ is a localization of $\ZZ$ satisfying condition \eqref{eq:assumption_DM}.

Then we can compute Borel-Moore motivic homology, for
 a $k$-variety $f:X \rightarrow \spec(k)$
 and a virtual bundle $v/X$ of rank $m$ as follows:
\begin{align*}
\HMBM_n(X,v,R)
&=\Hom_{\DAg{X,R}}\big(\un_X,f^!(\HM R) \otimes \Th(-)[-n]\big) \\
&=\Hom_{\DAg{X,R}}\big(\un_X,f^!(\gamma_*(\un_k)) \otimes \Th(-v)[-n]\big) \\
&\stackrel{(1)}=\Hom_{\DAg{X,R}}\Big(\un_X,\gamma_*\big(f^!(\un_k) \otimes M\Th(-v)[-n]\big)\Big) \\
&=\Hom_{\DMcdh(X,R)}\big(\un_X,f^!(\un_k) \otimes M\Th(-v)[-n]\big) \\
&\stackrel{(2)}=\Hom_{\DMcdh(X,R)}\big(\un_X,f^!(\un_k)(-m)[-2m-n]\big) \\
&\stackrel{(3)}=\mathrm{CH}_m(X,n) \otimes_\ZZ R
\end{align*}
where (1) follows from the properties mentioned in
 Paragraph \ref{num:motivic_singular}, (2) as $\DMcdh$ is oriented
 and (3) using \cite[Cor. 8.12]{CD5}.
\end{num}

\subsection{Generalized regulators}

\begin{num}\label{num:MW-regulators}
As explained in Paragraph \ref{num:morphisms},
 the commutativity of Diagram \eqref{eq:diagram_DMs} automatically
 induces morphisms of ring spectra as follows:
$$
\xymatrix@=20pt{
\HH_{\AA^1}R\ar^\psi[r] & \Ht R\ar^\varphi[r]\ar[d] & \HM R\ar[d] \\
& \Htet R\ar^{\varphi_\et}[r] & \HMet R.
}
$$
As explained in Paragraph \ref{num:morphisms_ringsp} and Remarks
 \ref{rem:morphisms_ringsp1}, \ref{rem:morphisms_ringsp2},
 these morphisms induce natural transformations of the four associated
 theories, compatible with products.

In particular, given a $k$-scheme $X$ and virtual bundle $v$ over $X$
 of rank $m$, we get morphisms
$$
\HH_{\AA^1}^n(X,v,R)
 \xrightarrow{\psi_*} \HH_{\mathrm{MW}}^n(X,v,R)
 \xrightarrow{\varphi_*} \HH_{\mathrm{M}}^{n+2m,m}(X,R)
$$
where the right hand side is Voevodsky's motivic cohomology
 (see \ref{num:mot_coh}).
 In brief, these maps are compatible with all the structures on
 cohomology described in Section \ref{sec:theories}.

Assume finally the conditions of Paragraph \ref{num:BM_motivic}
 are fulfilled. Then we get natural morphisms:
$$
\HH^{BM,\AA^1}_n(X,v,R)
 \xrightarrow{\psi_*} \HtBM_n(X,v,R)
 \xrightarrow{\varphi_*} \CH_m(X,n) \otimes_\ZZ R,
$$
compatible with contravariant and covariants functorialities,
 and localization long exact sequences
 (Paragraph \ref{num:localization_BM&c}).

Of course, when $X$ is a smooth $k$-variety, 
 the two maps $\varphi_*$ (resp. $\psi_*$)
 that appear above can be compared by duality
 (Theorem \ref{thm:duality}).
 Besides, if $v=[m]$ and $n=0$, one can check the
 latter map $\varphi_*$ is simply the canonical map:
$$
\wCH_m(X) \otimes_\ZZ R \longrightarrow \CH_m(X) \otimes_\ZZ R
$$
from Chow-Witt groups to Chow groups.
\end{num}
